\documentclass[sn-mathphys]{sn-jnl}

\usepackage{amsthm}

\usepackage{comment}
\usepackage{mathrsfs}

\jyear{2023}%

\theoremstyle{thmstyleone}%
\newtheorem{theorem}{Theorem}[section]
\newtheorem{proposition}[theorem]{Proposition}%
\newtheorem{lemma}[theorem]{Lemma}%
\newtheorem{corollary}[theorem]{Corollary}%
\newtheorem{remark}{Remark}%
\theoremstyle{thmstylethree}%

\theoremstyle{thmstyletwo}%

\renewcommand{\proofname}{Proof.}

\usepackage{anyfontsize}

\newcommand{\Complex}{ \mathbb{C} }

\newcommand{\Integer}{ \mathbb{Z} }

\newcommand{\Real}{ \mathbb{R} }
\newcommand{\RealPart}{ \mathrm{Re} }
\newcommand{\Set}[2]{\left \{#1 \, ; \, #2 \right \}}

\newcommand{\Torus}{\mathbb{T}}

\begin{document}

\title{Data Assimilation to the Primitive Equations in $H^2$}


\author*[1,2]{\fnm{} \sur{Ken Furukawa} (furukawa@sci.u-toyama.ac.jp)}



\affil*[1]{\orgdiv{Faculty of Science, Academic Assembly, University of Toyama, \orgaddress{\street{Gofuku 3190}, \city{Toyama}, \postcode{930-0887}, \state{Toyama}, \country{JAPAN}}}}




\abstract{In this paper we prove that the solution to the primitive equations is predicted by the corresponding data assimilation(DA) equations in $H^2$.
Although, the DA equation does not include the direct information about the base solution and its initial conditions, the solution to the DA equation exponentially convergence to the base(original) solution when the external forces are known even before they are observed.
Additionally, when the external force is not completely known but its spatially dense observations are available, then the DA is stable, $i.e.$ the DA solution lies in a sufficiently small neighborhood of the base solution.
}




\maketitle

\bmhead{Acknowledgments}
The author was partly supported by RIKEN Pioneering Project “Prediction for Science” and JSPS Grant-in-Aid for Young Scientists (No. 22K13948).

\section{Introduction} \label{intro}
We consider the data assimilation problem for the primitive equations.
The equations are given by
\begin{equation} \label{eq_primitive}
    \begin{alignedat}{3}
        \partial_t v - \Delta v + u \cdot \nabla v + \nabla_H \pi
        & = f
        & \quad \text{in} \quad
        & \Omega \times (0, \infty), \\
        \partial_3 \pi
        & = 0
        & \quad \text{in} \quad
        & \Omega \times (0, \infty), \\
        \mathrm{div} \, u
        & =0
        & \quad \text{in} \quad
        & \Omega \times (0, \infty), \\
        v(0)
        & =v_0
        & \quad \text{in} \quad
        & \Omega,
    \end{alignedat}
\end{equation}
where $T>0$, $u = (v, w) \in \Real^2 \times \Real$ is an unknown velocity field with initial data $v_0$, $\pi$ are unknown the pressure field.
The domain $\Omega = \Torus^2 \times (-l, 0)$ is the periodic layer.
The differential operators $\nabla_H = (\partial_1, \partial_2)^T$, $\mathrm{div}_H = \nabla_H \cdot$, and $\Delta_H = \nabla_H \cdot \nabla_H$ are the horizontal gradient, the horizontal divergence, and the horizontal Laplacian, respectively.
The vertical velocity $w$ is given by the formula
\begin{equation*}
    w (x^\prime, x_3, t)
    = - \int_{-l}^{x_3}
        \mathrm{div}_H \, v (x^\prime, z, t)
    dz.
\end{equation*}
For such $w$ the divergence-free condition
\begin{align*}
    \mathrm{div}_H v
    + \partial_3 w
    = 0
\end{align*}
is satisfied.
The boundaries $\Gamma_l$, $\Gamma_b$, and $\Gamma_u$ are correspond to lateral, bottom, upper boundaries ones.
We impose the periodicity for $u, \pi$ on the lateral boundaries and
\begin{align}\label{eq_bound_conditions}
    \begin{split}
        & \partial_3 v = 0,
        \quad w = 0
        \quad \textrm{on}
        \quad \Gamma_u,\\
        & v = 0,
        \quad w = 0
        \quad \textrm{on}
        \quad \Gamma_b.
    \end{split}
\end{align}
The primitive equations are a fundamental model of geophysical flows and have used in meteorology due to its low computational cost for simulations.
Data assimilation (DA) has been developed alongside the field of weather forecasting.
Since the spatial domain for meteorological forecasting is a stratified region where the vertical thickness is extremely small compared to the horizontal scale, it is reasonable to use primitive equations instead of the three-dimensional Navier-Stokes equations.

One of the earliest studies was conducted by Lions, Temam, and Wang \cite{LionsTemamWangShou1992}, where they considered the primitive equations in a shell domain.
The local well-posedness was proved by Guill\'{e}n-Gonz\'{a}lez, Masmoudi, and Rodr\'{i}guez-Bellido \cite{GuillenMasmoudiRodriguez2001}.
Cao and Titi \cite{CaoTiti2007} established the global well-posedness of the primitive equations in $H^1$.
Kukavica and Ziane \cite{KukavicaZiane2007} proved the existence of the global attractor in $L^2$-framework.
Hieber and Kashiwabara \cite{HieberKashiwabara2016} have extended the well-posedness for $L^p$-settings based on the theory of analytic semi-group.
To the author's knowledge, the most generated result is established by Giga, Gries, Hieber, Hussein, Kashiwabara \cite{GigaGriesHieberHusseinKashiwabara2020_analiticity}.
To consider the primitive equations in the framework of functional analysis, it is useful to use the equivalent equations such that
\begin{equation} \label{eq_primitive_evo}
    \begin{alignedat}{3}
        \partial_t v - P \Delta v + P (u \cdot \nabla v)
        &= P f
        &\quad \textrm{in} \quad
        &\Omega \times (0, \infty), \\
        \mathrm{div}_H \, \overline{v}
        &= 0
        &\quad \textrm{in} \quad
        &\Omega \times (0, \infty), \\
        v(0)
        &= v_0
        &\quad \textrm{in} \quad
        &\Omega,
    \end{alignedat}
\end{equation}
where $P: L^q(\Omega) \rightarrow L^q_{\overline{\sigma}}(\Omega)$ ($q \in (1,\infty)$) is the hydrostatic Helmholtz projection given by $P \varphi = \varphi + \nabla_H (- \Delta_H)^{-1} \mathrm{div}_H \overline{{\varphi}}$ and $\overline{\varphi} = \int_{-l}^0 \varphi(\cdot, \cdot, z) dz/l$ is the horizontal average.
We write $C^\infty(\Omega)$ to denote the set of horizontally periodic $C^\infty$-functions for the horizontal variable $x^\prime$.
The Lebesgue space $L^q(\Omega)$ and the Sobolev space $H^{s, q}(\Omega)$ are defined as the completion of $C^\infty(\Omega)$ by the standard $L^q$- and $H^{s, q}$-norms for $s \geq 0$ and $1 < q < \infty$.
We denote the $\overline{\mathrm{div}_H}$-free $L^q$-vector fields $L^q_{\overline{\sigma}}(\Omega)$ by
\begin{equation*}
    L^q_{\overline{\sigma}}(\Omega) 
    = \overline{
        \Set{
            \varphi \in C^\infty(\Omega)
        }{
            \mathrm{div}_H \overline{\varphi} = 0, \,
            \text{$\varphi$ satisfies (\ref{eq_bound_conditions})}
        }
    }^{\Vert \cdot \Vert_{L^q(\Omega)}}.
\end{equation*}
We analogously denote $\overline{\mathrm{div}_H}$-free Sobolev spaces by $H^{s, q}_{\overline{\sigma}}(\Omega)$.
Note that $P: L^q(\Omega) \rightarrow L^q_{\overline{\sigma}}(\Omega)$ is bounded.

We consider the DA equations associated with $v$ such as
\begin{equation} \label{eq_nudging}
    \begin{alignedat}{3}
        \partial_t \tilde{v} - P \Delta \tilde{v} + P \left(
            \tilde{u} \cdot \nabla \tilde{v}
        \right)
        &= g + \mu P (J_\delta v - J_\delta \tilde{v})
        & \quad \textrm{in} \quad
        &\Omega \times (0, \infty),\\
        \mathrm{div}_H \, \overline{\tilde{v}}
        &= 0
        & \quad \textrm{in} \quad
        &\Omega \times (0, \infty),\\
        \tilde{v}(0)
        &= \tilde{v}_0
        & \quad \textrm{in} \quad
        &\Omega,
    \end{alignedat}
\end{equation}
where $g$ is an external force, $\mu >0$ is a constant called an inflation parameter, $\delta >0$ is a constant representing the inverse of the observation density, and the bounded linear time-independent observation operator $J_\delta$ satisfies
\begin{align} \label{eq_J}
    \begin{split}
        \Vert
            J_\delta f
        \Vert_{L^q (\Omega)}
        & \leq C
        \Vert
            f
        \Vert_{L^q (\Omega)} \quad \textrm{for} \quad f \in L^q (\Omega), \\
        \Vert
            J_\delta f - f
        \Vert_{L^q (\Omega)}
        & \leq C \delta
        \Vert
            \nabla f
        \Vert_{L^q (\Omega)} \quad \textrm{for} \quad f \in H^{1, q} (\Omega),
    \end{split}
\end{align}
for some constant $C>0$.
The DA procedures with (\ref{eq_J}) are called the (continuous) nudging.
In this paper we consider the two cases
\begin{align*}
    g = f, J_\delta f.
\end{align*}
These cases correspond to when $f$ is perfectly detected and when $f$ is observed.
The DA can be considered as control problem to converge $\tilde{v}$ to $v$ as $t \rightarrow \infty$
The term $J_\delta v - J_\delta \tilde{v}$ in (\ref{eq_nudging}) behaves as a forcing term to make $\tilde{v}$ converge to $v$.
We remark that the information of $v$ itself is not used directly in the DA equations and the initial data is not used because we never obtain the perfect state of $v$ in the real world.

Azouani, Olson, and Titi \cite{AzouaniOlsonTiti2014} established a mathematical framework to consider the DA, and they showed the solution to the DA equations of the two-dimensional Navier-Stokes equation converges to the true solution in the energy space.
Albanez, Nussenzveig, and Titi \cite{AlbanezNussenzveigTiti2015} showed the same type of convergence result for the Navier-Stokes $\alpha$-model.
Pei \cite{Pei2019} showed the convergence of the DA equations to the primitive equations in the energy space.
The author \cite{Furukawa2024} proved the convergence in $L^p$-$L^q$- maximal regularity settings assuming fast decay of the base solution.
The fast decay is used to prove convergence in the Sobolev space and not used in convergence in $L^2$.
The just $L^2$-convergence at $t \rightarrow \infty$ implies the convergence in the same space as the observation $J_\delta v$.
We can know $v$ in $L^2$ if we have the accuracy of the observations is increased.
On the other hand, convergence in $H^m$ cannot be proven without using the equations, even if if we have a dense observation, which is one of the advantages of DA.
In this paper we prove the $H^2$ convergence without assuming decay of basic solution $v$.

We assume that $\mu \geq 1$ and $\delta \leq 1$ and
\begin{align}\label{eq_assumptions_for_v_mu_delta_for_theorem}
    \begin{split}
        & C_0 \sum_{m=1,2,4}
        \Vert
            v
        \Vert_{L^\infty(0, \infty; H^2(\Omega))}^m
        \leq \mu,\\
        & C_1 \mu \delta
        \leq 1,
    \end{split}
\end{align}
for some constant $C_j>0$ for $j=0,1$, which are just a some combinations of embedding constants and not depend onf $\mu$, $\delta$, and $v$.
Compare these bounds with (\ref{eq_assumptions_for_v_and_mu_for_energy_estimate}), (\ref{eq_condition_mu_and_delta_and_norm_of_v_1}) and (\ref{eq_condition_mu_and_delta_and_norm_of_v_2}).
We call these assumptions as (A).
The main result of this paper reads as follows.
\begin{theorem} \label{thm_main_thoerem1}
    Assume (A) and $g_\delta = f$.
    Let $v_0, V_0 \in H^1_{\overline{\sigma}}(\Omega)$ and $0 < \theta < \alpha$.
    Let $f$ satisfy
    \begin{align}\label{eq_assumptions_for_f_in_theorem}
        \begin{split}
            f
            & \in BC([0 , \infty); H^1(\Omega)^2)
            \cap C^\alpha((0 , \infty); H^1(\Omega)^2),\\
            & \partial_t f
            \in L^2_{loc}(0,T; L^2(\Omega)^2).
        \end{split}
    \end{align}
    Assume that the vector field $v$ is the solution to (\ref{eq_primitive}) obtained by \cite{CaoTiti2007} satisfying
    \begin{gather} \label{eq_bounds_for_v}
        \begin{split}
            & v
            \in BC^\theta (0, \infty; H^2_{\sigma}(\Omega)),\\
            & 
            \partial_t v 
            \in C^\theta (0, \infty; L^2_{\overline{\sigma}}(\Omega)),\\
        \end{split}
    \end{gather}
    Then there exists $\tilde{\mu} \geq \mu/4$, $T>0$ and a unique solution
    \begin{align*}
        \tilde{v}
        \in C^1(0, \infty; L^2(\Omega)^2)
        \cap C(0, \infty; H^2(\Omega))^2
    \end{align*}
    to (\ref{eq_nudging}) such that
    \begin{align} \label{eq_decay_estimate_in_main_theorem}
        \begin{split}
            \Vert
                \partial_t (v(t) - \tilde{v}(t))
            \Vert_{L^2(\Omega)^2}
            + \Vert
                v(t) - \tilde{v}(t)
            \Vert_{H^2_{\overline{\sigma}}(\Omega)}
            \leq C e^{ - \tilde{\mu} t}.
        \end{split}
    \end{align}
    for all $t > T$.
\end{theorem}
\begin{theorem} \label{thm_main_thoerem2}
    Assume (A) and $g_\delta = J_\delta f$.
    Under the same assumptions for $v_0, V_0, f$,
    Then there exists $\tilde{\mu} \geq \mu/4$, $T>0$, and a unique solution
    \begin{align*}
        \tilde{v}
        \in C^1(0, \infty; L^2(\Omega)^2)
        \cap C(0, \infty; H^2(\Omega))^2
    \end{align*}
    to (\ref{eq_nudging}) such that
    \begin{align} \label{eq_decay_estimate_in_main_theorem2}
        \begin{split}
            & \Vert
                \partial_t (v(t) - \tilde{v}(t))
            \Vert_{L^2(\Omega)^2}
            + \Vert
                v(t) - \tilde{v}(t)
            \Vert_{H^2_{\overline{\sigma}}(\Omega)}\\
            & \leq C \delta \limsup_{t \rightarrow \infty} \Vert
                \nabla f(t)
            \Vert_{L^2(\Omega)^2}.
        \end{split}
    \end{align}
    for $t > T$.
\end{theorem}
\begin{remark} \label{rem_for_main_theorem}
    \begin{itemize}
        \item Such basic solution $v$ exists, $c.f.$ Proposition 5.8 in \cite{HieberKashiwabara2016} for $\theta \in (0, 1/4)$.
        
        \item The existence of the solution to (\ref{eq_diff}) in $BC(0, T; H^1(\Omega))$ for all $T \in (0,\infty)$ with initial data $V_0 \in H^1(\Omega)$ and satisfying (\ref{eq_assumptions_for_f_in_theorem}) has been proved by the author in \cite{Furukawa2024} by a priori estimates.
        
        \item The rate $\tilde{\mu}$ is deduced from (\ref{eq_lower_bound_for_B_t_varphi_varphi}) and an estimate for the linear evolution operator Lemma \ref{lem_maximal_regularity_estimate_in_F_nu_beta_sigma}.
        
        \item Theorem \ref{thm_main_thoerem1} is stronger than Pei's result because the topology involved is stronger than $L^2$.
        Moreover, the assumptions of Theorem \ref{thm_main_thoerem1} are weaker than those in \cite{Furukawa2024}.

        \item The $H^2$-norm of $v$ appears in assumption (A).
        This norm can be dominated by polynomials of $C^\alpha (0, \infty; L^2(\Omega)^2)$-norm for $f$.
    \end{itemize}
\end{remark}
In the proof, we employ the theory of maximal regularity in temporal H\"{o}lder spaces for the equations of the difference
\begin{align} \label{eq_diff}
    \begin{split}
        &\partial_t V - P \Delta V
        - P(U \cdot \nabla V) \\
        & \quad\quad\quad
        + P(u \cdot \nabla V)
        + P(U \cdot \nabla v)
        = P F_\delta - \mu P J_\delta V
        \quad \textrm{in}
        \quad \Omega \times (0, \infty), \\
        & \mathrm{div}_H \, \overline{V}
        = 0
        \quad \textrm{in}
        \quad \Omega \times (0, \infty),\\
        & V(0)
        = V_0
        \quad \textrm{in}
        \quad \Omega,
    \end{split}
\end{align}
for $V = v - \tilde{v}$, $W = w - \tilde{w}$, $U = (V, W)$, $V_0 = v_0 - \tilde{v}_0$, and $F_\delta = f - g$.
The solution $\tilde{v}$ to (\ref{eq_nudging}) is given as $v - V$.
Note that $F_\delta = 0 \text{ or } f - J_\delta f$.
We generate an evolution operator $T_{\mu, \delta}(t, s)$ corresponding to the linearized operator
\begin{align*}
    \mathcal{A}_{\mu, \delta}(t) V
    = \mu P J_\delta V
    - P \Delta V
    + P(u(t) \cdot \nabla V)
    + P(U \cdot \nabla v(t)),
\end{align*}
associated the same domain as the hydrostatic Stokes operator $A = - P \Delta$ such that
\begin{align*}
    D(\mathcal{A}_{\mu, \delta}(t))
    = D(A)
    = \{ \varphi \in H^2(\Omega)^2 \cap L^2_{\overline{\sigma}}(\Omega)
        ;
        \text{$\varphi$ satisfies (\ref{eq_bound_conditions}).}
    \}.
\end{align*}
The domain is independent of $t$ and the perturbations in $\mathcal{A}_{\mu, \delta}(t)$ make sense because $v(t) \in H^2(\Omega)^2$.
The previous result \cite{Furukawa2024} is based on the linear theory for the perturbed hydrostatic Stokes operator $\mu P J_\delta + A$.
To estimate the cross terms $P(u \cdot \nabla V) + P(U \cdot \nabla v)$, the decay of $v$ was used.
Now the cross terms are included in $\mathcal{A}_{\mu, \delta}(t)$ and are dominated by $\mu I + A$.
Therefore we does not assume fast decay of $v$.

We have already established the global well-posedness of (\ref{eq_diff}) under the same assumption of Theorem \ref{thm_main_thoerem1} in \cite{Furukawa2024}.
We shall prove the decay $V(t) \rightarrow 0$.
Since $J_\delta = I + K_\delta$, the operator $\mathcal{A}_{\mu, \delta}(t)$ is decomposition into
\begin{align*}
    \mathcal{A}_{\mu, \delta}(t) V
    = (
        \mu P J_\delta
        - A
    )V
    + (
        \mu P K_\delta V
        + P(u(t) \cdot \nabla V)
        + P(U \cdot \nabla v(t))
    ).
\end{align*}
The second term in the right-hand side is dominated by the first term if $v$ is $\theta$-H\"{o}lder continuous for sufficiently large $\mu>0$ and small $\delta$.
To construct $T_{\mu, \delta}(t, s)$, the H\"{o}lder continuity is required.
We also observe the exponential decay of $T(t, s)$ as $e^{- \tilde{\mu} (t - s)}$ for some $\tilde{\mu} > 0$.

To show the decay and stability we first show an energy estimate.
The energy estimate implies that $V(t_0)$ can be small in $H^1$ for sufficiently large $t_0$.
We next construct the evolution operator $T_{\mu, \delta}(t,s)$ associated with $\mathcal{A}_{\mu, \delta}(t)$ by the Tanabe theory, the reader is referred to \cite{Tanabebook1979,Tanabe1960,Kato1953,KatoTanabe1962,Yagi2010}.
The operator $T_{\mu, \delta}(t,s)$ have exponential decay properties for large $\mu$ and small $\delta$ satisfying assumptions in Theorem \ref{thm_main_thoerem1}-\ref{thm_main_thoerem2}.
We also show $T_{\mu, \delta}(t,s)$ have the maximal H\"{o}lder regularity.
Combining with the formula
\begin{align*}
    V(t)
    = T_{\mu, \delta}(t, t_0) V(t_0)
    + \int_{t_0}^t
        T_{\mu, \delta}(t, s) P (
            U(s) \cdot \nabla V(s)
            + F_\delta(s)
        )
    ds, \quad
    t > t_0
\end{align*}
and the Fujita-Kato fixed point argument, we deduce that $V(t)$ with initial data $V(t_0)$ tends to zero exponentially as $t \rightarrow \infty$ if $F_\delta = 0$.
If $F_\delta = f - J_\delta f$, by the same argument, we obtain the stability.
Note that $v$ and $f$ are not smooth functions in general, by this reason we cannot see the converge and the stability in more smooth space than Theorem \ref{thm_main_thoerem1}-\ref{thm_main_thoerem2} without assuming more smoothness for vector fields $v$ and $f$ and boundedness in more regular spaces for the observation operator $J_\delta$.

In this paper we use the following notation and notion.
We denote by $L^q_x L^r_{x^\prime} (\Omega \times \Omega^\prime) = L^q(\Omega; L^r(\Omega^\prime))$ the mixed Lebesgue space for $1 \leq q, r \leq \infty$, domains $\Omega, \Omega^\prime$, and $x \in \Omega, x^\prime \in \Omega^\prime$.
We also use this kind of notation for mixed Sobolev and Sobolev-Lebesgue spaces.
We write
\begin{equation*}
    \dot{H}^{m, q}(\Omega)
    = \left \{
        \varphi \in L^1_{loc}(\Omega)
        \, ; \,
        \Vert \varphi \Vert_{\dot{H}^{m}}
        := \sum_{\vert \alpha \vert = m} \Vert \partial^\alpha_x \varphi \Vert_{L^q(\Omega)} < \infty
    \right \},
\end{equation*}
for $m \in \Integer_{\geq 0}$ and $q \in (1, \infty)$.
For a Banach space $X$ and $r>0$, we denote by $B_X(r)$ the ball in $X$ with radius $r$.
The denote by $\Sigma_{\theta, a}$ the sector
\begin{align*}
    \Sigma_{\theta, a}
    = \{ \lambda + a \in \Complex
        ;
        \vert
            \arg \lambda
        \vert
        \leq \theta
    \}
\end{align*}
for $\theta \in (0, \pi)$ and $a \in \Real$.
We denote $\Gamma_{\theta, a} = \partial \Sigma_{\theta, a}$
When $a = 0$, we simply denote by $\Sigma_{\theta}$ and $\Gamma_\theta$.

The paper is organized as follows.
In Section \ref{section_a_priori_estimate}, we establish some a priori estimates.
We find a sequence of large $t_j>0$ such that $V(t_j)$ is small in $H^1$
In Section \ref{section_perturbed_hydrostatic_stokes_operator}, we observe some properties of the perturbed hydrostatic operator $\mathcal{A}_{\mu, \delta}(t)$.
We will construct an evolution operator $T_{\mu, \delta}(t, s)$ associated with $\mathcal{A}_{\mu, \delta}(t)$ by the Tanabe theory.
In Section \ref{section_linear_evolution_operator}, we establish the maximal regularity results in a temporal H\"{o}lder-framework, $c.f.$ the book by Yagi \cite{Yagi2010}.
For an introduction to semigroup theory, which is also used in the theory of evolution operators, the reader is referred to \cite{EngelNagel2000,Sohr2001}.
In this section we also observe dependence of $\Vert v \Vert_{C^\theta(0, \infty; H^2(\Omega)^2)}$ for the index of the exponential function and some constants of inequalities.
In Section \ref{section_nonlinear_problems}, we apply the linear theory to the non-linear problems.
We construct the solution to (\ref{eq_estimate_for_M_mu_delta_in_F_sigma_half}) by Fujita-Kato fixed point arguments, at the same time we obtain the exponential decay and stability in $H^2$-framework.

\section{A priori estimate} \label{section_a_priori_estimate}
We first establish some a priori estimates, see also a priori estimates in Section 6 of \cite{HieberKashiwabara2016}.
Let $V_0 \in H^1_{\overline{\sigma}}(\Omega)$.
We assume that
\begin{align*}
    \sup_{t>0} \Vert
        v(t)
    \Vert_{H^2(\Omega)}
    & < \infty,\\
    \int_t^{t+R}
        \Vert
            \nabla^3 v(t)
        \Vert_{H^2(\Omega)}^2
    ds
    & < \infty,
\end{align*}
for some $R>0$.
For a divergence-free smooth vector field $\varphi$, it is known that
\begin{align*}
    \int_{\Omega}
        \varphi \cdot \nabla V \cdot V
    dx
    & = - \int_\Omega
       \sum_{j,k} \partial_j \varphi_j V_k V_k
       + \varphi_j V_k \partial_j V_k
    dx\\
    & = - \int_\Omega
        \varphi \cdot \nabla V \cdot V
    dx,
\end{align*}
therefore
\begin{align} \label{eq_divfree_cancelation}
    \int_{\Omega}
        (\varphi \cdot \nabla V) \cdot V
    dx
    = 0.
\end{align}

\begin{proposition} \label{prop_trigonal_estimate}
    Let $p \in (1, \infty)$ and $r_1, r_2 \in (1 - 1/p, \infty)$ satisfy $1 - 1/p \geq 1/r_1 + 1/r_2$.
    Then
    \begin{align*}
        &\left\vert
            \int_\Omega
                \left(
                    \int_{-l}^{x_3}
                        f
                    dz
                \right)
                g h
            dx
        \right\vert \\
        & \leq C \Vert
            f
        \Vert_{L^2(\Omega)}^{\frac{2}{p}}
        \Vert
            f
        \Vert_{H^1_{x^\prime}L^2_{x_3}(\Omega)}^{1 - \frac{2}{p}}
        \Vert
            g
        \Vert_{L^2(\Omega)}^{\frac{2}{r_1}}
        \Vert
            g
        \Vert_{H^1_{x^\prime}L^2_{x_3}(\Omega)}^{1 - \frac{2}{r_1}}
        \Vert
            h
        \Vert_{L^2(\Omega)}^{\frac{2}{r_2}}
        \Vert
            h
        \Vert_{H^1_{x^\prime}L^2_{x_3}(\Omega)}^{1 - \frac{2}{r_2}}
    \end{align*}
    for some constant $C>0$ and all $f, g, h \in H^1(\Omega)$.
\end{proposition}
\begin{proof}
    See Proposition 11 in \cite{Furukawa2024} or Lemma 2.1 in \cite{CaoTiti2017} for the proof of this proposition.
\end{proof}
We apply integration by parts, the H\"{o}lder inequality, Proposition \ref{prop_trigonal_estimate} with $p=2,r_1=r_2=4$, and the formula (\ref{eq_divfree_cancelation}) to (\ref{eq_diff}) to see that
\begin{align*}
    & \frac{\partial_t}{2} \Vert
        V
    \Vert_{L^2 (\Omega)}^2
    + \Vert
        \nabla V
    \Vert_{L^2(\Omega)}^2 \\
    & \leq \left \vert
        \int_\Omega
            (U \cdot \nabla v) \cdot V
            + F \cdot V
        dx
    \right \vert
    - \int_\Omega
        \mu J_\delta V \cdot V
    dx\\
    & = \left \vert
        \int_\Omega
            \left[
                (V \cdot \nabla_H v) \cdot V
                - \left(
                    \int_{-l}^{x_3}
                        \mathrm{div}_H V
                    d\zeta
                \right)
                \partial_3 v \cdot V
                + F \cdot V
            \right]
        dx
    \right \vert\\
    & - \int_\Omega
        \mu J_\delta V \cdot V
    dx\\
    & \leq C \Vert
        V
    \Vert_{L^2(\Omega)}
    \Vert
        \nabla_H v
    \Vert_{L^6(\Omega)}
    \Vert
        V
    \Vert_{L^3(\Omega)} \\
    & + C \Vert
        \nabla_H V
    \Vert_{L^2(\Omega)}
    \Vert
        \partial_3 v
    \Vert_{L^2(\Omega)}^{\frac{1}{2}}
    \Vert
        \partial_3 v
    \Vert_{H^1(\Omega)}^{\frac{1}{2}}
    \Vert
        V
    \Vert_{L^2(\Omega)}^{\frac{1}{2}}
    \Vert
        V
    \Vert_{H^1(\Omega)}^{\frac{1}{2}} \\
    & + C \Vert
        F
    \Vert_{L^2(\Omega)}
    \Vert
        \nabla V
    \Vert_{L^2(\Omega)}
    - \mu \Vert
        V
    \Vert_{L^2(\Omega)}^2
    + C \mu \delta \Vert
        \nabla V
    \Vert_{L^2(\Omega)}^2,
\end{align*}
where we used the estimate
\begin{align}
    \begin{split}
        & - \int_\Omega
            J_\delta V \cdot V
        dx
        = - \int_\Omega
            (
                V
                + K_\delta V
            )\cdot V
        dx\\
        & = - \Vert
            V
        \Vert_{L^2(\Omega)}^2
        - \int_\Omega
            K_\delta V \cdot V
        dx\\
        & \leq - \Vert
            V
        \Vert_{L^2(\Omega)}^2
        + C \delta \Vert
            \nabla V
        \Vert_{L^2(\Omega)}
        \Vert
            V
        \Vert_{L^2(\Omega)}\\
        & \leq - \Vert
            V
        \Vert_{L^2(\Omega)}^2
        + C \delta \Vert
            \nabla V
        \Vert_{L^2(\Omega)}^2
    \end{split}
\end{align}
for some constant $C>0$.
Using the interpolation inequality and the embedding
\begin{equation*}
    \Vert
        \varphi
    \Vert_{L^3(\Omega)}
    \leq C \Vert
    \varphi
    \Vert_{L^2(\Omega)}^{\frac{1}{2}}
    \Vert
        \varphi
    \Vert_{\dot{H}^1(\Omega)}^{\frac{1}{2}} ,\quad
    \Vert
        \varphi
    \Vert_{L^6(\Omega)}
    \leq C \Vert
        \varphi
    \Vert_{H^1(\Omega)}
\end{equation*}
for $\varphi \in H^1(\Omega)$, the inequality (\ref{eq_J}) for $F$, and the Young inequality, we find that there exists a small $\delta$ such that
\begin{equation} \label{eq_differential_inequality_a_priori_estimate}
    \begin{aligned}
        & \frac{\partial_t}{2} \Vert
            V
        \Vert_{L^2 (\Omega)}^2
        + \Vert
            \nabla V
        \Vert_{L^2(\Omega)}^2 \\
        & \leq C \Vert
            V
        \Vert_{L^2(\Omega)}^{3/2}
        \Vert
            \nabla_H v
        \Vert_{L^6(\Omega)}
        \Vert
            \nabla V
        \Vert_{L^3(\Omega)}^{1/2} \\
        & + C \Vert
            \nabla V
        \Vert_{L^2(\Omega)}^{3/2}
        \Vert
            \partial_3 v
        \Vert_{L^2(\Omega)}^{\frac{1}{2}}
        \Vert
            \partial_3 v
        \Vert_{H^1(\Omega)}^{\frac{1}{2}}
        \Vert
            V
        \Vert_{L^2(\Omega)}^{\frac{1}{2}} \\
        & +  C \Vert
            F
        \Vert_{L^2(\Omega)}
        \Vert
            \nabla V
        \Vert_{L^2(\Omega)}
        - \mu \Vert
            V
        \Vert_{L^2(\Omega)}^2
        + C \mu \delta \Vert
            \nabla V
        \Vert_{L^2(\Omega)}^2\\
        & \leq C (
            \Vert
                \nabla_H v
            \Vert_{H^1(\Omega)}^{4/3}
            + \Vert
                \nabla v
            \Vert_{L^2(\Omega)}
            \Vert
                \nabla v
            \Vert_{H^1(\Omega)} \\
        &   \quad + \Vert
                \nabla v
            \Vert_{L^2(\Omega)}^2
            \Vert
                \nabla v
            \Vert_{H^1(\Omega)}^2
        )
        \Vert
            V
        \Vert_{L^2(\Omega)}^2 \\
        & - \mu
        \Vert
            V
        \Vert_{L^2(\Omega)}^2
        + C \delta^2 \Vert
            \nabla f
        \Vert_{L^2(\Omega)}^2
        + \frac{1}{2} \Vert
            \nabla V
        \Vert_{L^2(\Omega)}^2.
    \end{aligned}
\end{equation}
Since $v \in L^\infty (0, \infty; H^2(\Omega))$, we take $\mu > 0$ so large that
\begin{align} \label{eq_assumptions_for_v_and_mu_for_energy_estimate}
    \begin{split}
        & C (
            \Vert
                \nabla_H v
            \Vert_{H^1(\Omega)}^{4/3}
            + \Vert
                \nabla v
            \Vert_{L^2(\Omega)}
            \Vert
                \nabla v
            \Vert_{H^1(\Omega)}\\
        & \quad\quad\quad\quad\quad\quad
            + \Vert
                \nabla v
            \Vert_{L^2(\Omega)}^2
            \Vert
                \nabla v
            \Vert_{H^1(\Omega)}^2
        )
        \leq \mu/2,
    \end{split}
\end{align}
this corresponds to the first inequality of (\ref{eq_assumptions_for_v_mu_delta_for_theorem}), and apply the Gronwall inequality to obtain
\begin{align*}
    & \Vert
        V (t)
    \Vert_{L^2(\Omega)}^2\\
    & \leq e^{- \mu (t - \tau)/2} \Vert
        V(\tau)
    \Vert_{L^2(\Omega)}^2
    + C \delta^2 \int_\tau^t
        e^{- \mu (t - s)/2}
        \Vert
            \nabla f (s)
        \Vert_{L^2(\Omega)}^{2}
    ds\\
    & \leq e^{- \mu (t - \tau)/2} \Vert
        V_0
    \Vert_{L^2(\Omega)}^2
    + \frac{C \delta^2}{\mu} \sup_{\tau < s < t}
    \Vert
        \nabla f (s)
    \Vert_{L^2(\Omega)}^{2}
    (
        1
        - e^{-\mu (t - \tau)/2}
    ),
\end{align*}
for all $0 \leq \tau < t$ and
\begin{align*}
    \limsup_{t \rightarrow \infty} \Vert
        V(t)
    \Vert_{L^2(\Omega)}
    \leq \frac{C\delta^2}{\mu} \limsup_{t \rightarrow \infty}
    \Vert
        \nabla f (t)
    \Vert_{L^2(\Omega)}^2.
\end{align*}
We also see from (\ref{eq_differential_inequality_a_priori_estimate}) that
\begin{align*}
    & \int_t^{t+R}
        \Vert
            \nabla V(s)
        \Vert_{L^2(\Omega)}^2
    ds\\
    & \leq
    C (
        \Vert
            V (t + R)
        \Vert_{L^2 (\Omega)}^2
        - \Vert
            V (t)
        \Vert_{L^2 (\Omega)}^2
    )\\
    & + \frac{C \delta^2}{\mu}
    \int_t^{t+R}
        \Vert
            \nabla f (s)
        \Vert_{L^2(\Omega)}^2
    ds\\
    & \leq
    C \Vert
        V (t + R)
    \Vert_{L^2 (\Omega)}^2
    + \frac{C\delta^2 R}{\mu}
    \sup_{t < s < t + R}
        \Vert
            \nabla f (s)
        \Vert_{L^2(\Omega)}^2\\
    & \leq C e^{- \mu t/2} \Vert
        V_0
    \Vert_{L^2(\Omega)}^2
    + \frac{C \delta^2 R}{\mu} \sup_{t < s < t + R}
    \Vert
        \nabla f (s)
    \Vert_{L^2(\Omega)}^{2}
    (1 - e^{-\mu t/2}).
\end{align*}
Therefore, we deduce that
\begin{align} \label{eq_estimate_for_tempral_average_of_V}
    \int_n^{n+1}
        \Vert
            \nabla V(s)
        \Vert_{L^2(\Omega)}^2
    ds
    \leq \varepsilon
    + \frac{C \delta^2}{\mu} \limsup_{s \rightarrow \infty}
    \Vert
        \nabla f (s)
    \Vert_{L^2(\Omega)}^{2}.
\end{align}
for small $\varepsilon$, some $C>0$, and all integer $n > N_0$ for some sufficiently large $N$.
By \cite{Furukawa2024} we see that $V \in BC(0,T; H^1(\Omega))$ for all $0 < T < \infty$.
Therefore, combining this fact and (\ref{eq_estimate_for_tempral_average_of_V}), we deduce that for each small $\varepsilon >0$ there exists $t_n \in [n, n+1]$ for $n > N_0$ such that
\begin{align} \label{eq_definition_of_t0}
    \Vert
        \nabla V(t_n)
    \Vert_{L^2(\Omega)}^2
    \leq \varepsilon
    + \frac{C \delta^2}{\mu} \limsup_{s \rightarrow \infty}
    \Vert
        \nabla f (s)
    \Vert_{L^2(\Omega)}^{2}.
\end{align}

\section{Perturbed hydrostatic Stokes operator} \label{section_perturbed_hydrostatic_stokes_operator}
\subsection{Evolution operators}
\begin{proposition}[See \cite{HieberKashiwabara2016}] \label{prop_estimate_for_convection_term}
    Let $s\geq0$ and $1 < q < \infty$.
    There exists a constant $C>0$ such that
    \begin{align*}
        & \Vert
            v_1 \cdot \nabla_H v_2
        \Vert_{H^{s, q}(\Omega)}
        + \left \Vert
            \int_{-l}^{x_3}
                \mathrm{div}_H v_1
            dz \,
            \partial_3 v_2
        \right \Vert_{{H^{s, q}(\Omega)}} \\
        & \leq C \Vert
            v_1
        \Vert_{H^{s + 1 + \frac{1}{q}, q}(\Omega)}
        \Vert
            v_2
        \Vert_{H^{s + 1 + \frac{1}{q}, q}(\Omega)}
    \end{align*}
    for all $v_1, v_2 \in H^{s + 1 + \frac{1}{q}, q}(\Omega)^2$.
\end{proposition}

We denote the linearized perturbed operator $\mathcal{A}_{\mu, \delta}(t)$ by
\begin{gather} \label{eq_definition_of_A_mu_delta}
    \begin{split}
        \mathcal{A}_{\mu, \delta}(t) \varphi^\prime
        := - P \Delta \varphi^\prime
        + P(u \cdot \nabla \varphi^\prime)
        + P(\varphi \cdot \nabla v)
        + \mu P J_\delta \varphi^\prime,\\
        D(\mathcal{A}_{\mu, \delta}(t))
        = \{
            \varphi^\prime \in H^2(\Omega)^2
            :
            \mathrm{div}_H \overline{\varphi^\prime} = 0,
            \text{$\varphi^\prime$ satisfies (\ref{eq_bound_conditions})}
        \},
    \end{split}
\end{gather}
where we have denoted as
\begin{align} \label{eq_manner_varphi_prime_to_varphi}
    \varphi
    = (\varphi^\prime, \varphi_3), \,
    \varphi^\prime
    \in H^1(\Omega)^2, \,
    \varphi_3
    = - \int_{-1}^{x_3} \mathrm{div}_H \varphi^\prime dz.
\end{align}
Note that $\mathcal{A}_{\mu, \delta}(t)$ depends on $t$ since $v$ depends on $t$ but the domain $D(\mathcal{A}_{\mu, \delta}(t))$ is independent of $t$.
We denote the linear closed operator $\mathcal{B}_{\mu, \delta}(t)$ by
\begin{align*}
    \mathcal{B}_{\mu, \delta}(t) \varphi^\prime
    & := P(u \cdot \nabla \varphi^\prime)
    + P(\varphi \cdot \nabla v)
    + \mu P K_\delta \varphi^\prime,\\
    D(\mathcal{B}_{\mu, \delta}(t))
    & = D(A^{3/2})
    \hookrightarrow D(A).
\end{align*}
The domain $D(\mathcal{B}_{\mu, \delta}(t))$ is independent of $t$.
We set the bilinear form $B$ by
\begin{align*}
    B_{\mu, \delta}(t; \varphi^\prime, \psi^\prime)
    := \int_\Omega
        \mu \varphi^\prime \cdot \psi^\prime
        + \nabla \varphi^\prime : \nabla \psi^\prime
        + \mathcal{B}_{\mu, \delta}(t) \varphi^\prime \cdot \nabla \psi^\prime
    dx.
\end{align*}
for $\varphi^\prime, \psi^\prime \in X_{1/2}$, where $X_{1/2}$ is defined by
\begin{align*}
    X_{1/2}
    = \{ \varphi^\prime \in H^1(\Omega)^2
        ;
        \int_{-l}^0 \mathrm{div}_H \varphi^\prime dz
        = 0, \,
        \varphi^\prime \vert_{x_3=-l}
        = 0
    \}.
\end{align*}

Since $v \in BC(0, T; H^2)$, we deduce form the Cauchy-Schwarz inequality and the Poincar\'{e} inequality that
\begin{align} \label{eq_bound_for_B_by_nabla_phi_nabla_psi}
    \begin{split}
        & \vert
            B_{\mu, \delta}(t; \varphi^\prime, \psi^\prime)
        \vert\\
        & \leq \mu \Vert
            \varphi^\prime
        \Vert_{L^2(\Omega)}
        \Vert
            \psi^\prime
        \Vert_{L^2(\Omega)}
        + \Vert
            \nabla \varphi^\prime
        \Vert_{L^2(\Omega)}
        \Vert
            \nabla \psi^\prime
        \Vert_{L^2(\Omega)}\\
        & + C
        \Vert
            v(t)
        \Vert_{H^2(\Omega)}
        \Vert
            \nabla \varphi^\prime
        \Vert_{L^2(\Omega)}
        \Vert
            \nabla \psi^\prime
        \Vert_{L^2(\Omega)}
        + C \mu \delta \Vert
            \nabla \varphi^\prime
        \Vert_{L^2(\Omega)}
        \Vert
            \psi^\prime
        \Vert_{L^2(\Omega)}\\
        & \leq C  \Vert
            \nabla \varphi^\prime
        \Vert_{L^2(\Omega)}
        \Vert
            \nabla \psi^\prime
        \Vert_{L^2(\Omega)}
    \end{split}
\end{align}
for some $C>0$ which depends on $\mu, \delta, \Vert v \Vert_{L^\infty_t H^2_x(Q_T)}$.
We see from these estimates that the bilinear $B(t, \cdot, \cdot): H^1_{\overline{\sigma}}(\Omega) \times H^1_{\overline{\sigma}}(\Omega) \rightarrow \Real$ is bounded.
We deduce that
\begin{align*}
    & B_{\mu, \delta}(t; \varphi^\prime, \varphi^\prime)\\
    & \geq \mu \Vert
        \varphi^\prime
    \Vert_{L^2(\Omega)}^2
    + \Vert
        \nabla \varphi^\prime
    \Vert_{L^2(\Omega)}^2
    - \left \vert
        \int_\Omega
            P (\varphi \cdot \nabla v) \cdot \varphi^\prime
        dx
    \right \vert
    - \mu \left \vert
        \int_\Omega
            P K_\delta \varphi^\prime \cdot \varphi^\prime
        dx
    \right \vert\\
    & \geq \mu \Vert
        \varphi^\prime
    \Vert_{L^2(\Omega)}^2
    + \Vert
        \nabla \varphi^\prime
    \Vert_{L^2(\Omega)}^2\\
    & - C \Vert
        v
    \Vert_{L^\infty_t H^2_x(\Omega)}
    \Vert
        \nabla \varphi^\prime
    \Vert_{L^2(\Omega)}^{3/2}
    \Vert
        \varphi^\prime
    \Vert_{L^2(\Omega)}^{1/2}
    - C \mu \delta \Vert
        \nabla \varphi^\prime
    \Vert_{L^2(\Omega)}
    \Vert
        \varphi^\prime
    \Vert_{L^2(\Omega)}
\end{align*}
for some constant $C>0$.
Since
\begin{align*}
    & \Vert
        v
    \Vert_{L^\infty_t H^2_x(\Omega)}
    \Vert
        \nabla \varphi^\prime
    \Vert_{L^2(\Omega)}^{3/2}
    \Vert
        \varphi^\prime
    \Vert_{L^2(\Omega)}^{1/2}\\
    & \leq C \Vert
        v
    \Vert_{L^\infty_t H^2_x(\Omega)}^4
    \Vert
        \varphi^\prime
    \Vert_{L^2(\Omega)}^2
    + \frac{1}{4} \Vert
        \nabla \varphi^\prime
    \Vert_{L^2(\Omega)}^2,
\end{align*}
we obtain that
\begin{align*}
    & B_{\mu, \delta}(t; \varphi^\prime, \varphi^\prime)\\
    & \geq \left(
        \mu
        - C_1 \Vert
            v
        \Vert_{L^\infty_t H^2_x(\Omega)}^4
        - C_2 \mu^2 \delta^2
    \right) \Vert
        \varphi^\prime
    \Vert_{L^2(\Omega)}^2
    + \frac{1}{2} \Vert
        \nabla \varphi^\prime
    \Vert_{L^2(\Omega)}^2
\end{align*}
for some constant $C>0$.
If we take $\mu$ so large and take $\delta$ so small that
\begin{align} \label{eq_condition_mu_and_delta_and_norm_of_v_1}
    \frac{\mu}{4}
    \geq C_1 \Vert
        v
    \Vert_{L^\infty_t H^2_x(\Omega)}^4, \quad
    \frac{\mu}{4}
    \geq C_2 \mu^2 \delta^2,
\end{align}
we can conclude that
\begin{align} \label{eq_lower_bound_for_B_t_varphi_varphi}
    B_{\mu, \delta}(t; \varphi^\prime, \varphi^\prime)
    \geq \mu_\ast \Vert
        \varphi^\prime
    \Vert_{L^2(\Omega)}^2
    + \frac{1}{2} \Vert
        \nabla \varphi^\prime
    \Vert_{L^2(\Omega)}^2
\end{align}
for some $\mu_\ast \geq \mu/2 > 0$.
Typically we can take $\delta = O(\mu^{-1})$ for large $\mu>0$.
The bilinear $B_{\mu, \delta}$ is coercive.
By the Sobolev inequalities and $D(A^{3/2}) \hookrightarrow H^{3/2}(\Omega)^2$, we have
\begin{align} \label{eq_estimate_for_mathcal_B_part1}
    \begin{split}
        & \Vert
            u \cdot \nabla V
        \Vert_{L^2(\Omega)}\\
        & \leq \Vert
            v \cdot \nabla_H V
        \Vert_{L^2(\Omega)}
        + \Vert
            w \partial_3 V
        \Vert_{L^2(\Omega)}\\
        & \leq C \left(
            \Vert
                v
            \Vert_{H^{3/2+\varepsilon}(\Omega)}
            \Vert
                V
            \Vert_{H^1(\Omega)}
            + \int_{-1}^1
                \Vert
                    \mathrm{div}_H v
                \Vert_{L^4(\Torus^2)}
            dz
            \Vert
                \partial_3 V
            \Vert_{L^2_{x_3}L^4_{x^\prime}(\Omega)}
        \right)\\
        & \leq C \left(
            \Vert
                v
            \Vert_{H^{3/2+\varepsilon}(\Omega)}
            \Vert
                V
            \Vert_{H^1(\Omega)}
            + \Vert
                v
            \Vert_{H^{3/2}(\Omega)}
            \Vert
                V
            \Vert_{H^{3/2}(\Omega)}
        \right)
    \end{split}
\end{align}
for all $\varepsilon>0$, and
\begin{align} \label{eq_estimate_for_mathcal_B_part2}
    \begin{split}
        & \Vert
            U \cdot \nabla v
        \Vert_{L^2(\Omega)}\\
        & \leq \Vert
            V \cdot \nabla_H v
        \Vert_{L^2(\Omega)}
        + \Vert
            W \partial_3 v
        \Vert_{L^2(\Omega)}\\
        & \leq C \left(
            \Vert
                V
            \Vert_{H^1(\Omega)}
            \Vert
                v
            \Vert_{H^{3/2}(\Omega)}
            + \int_{-1}^1
                \Vert
                    \mathrm{div}_H V
                \Vert_{L^4(\Torus^2)}
            dz
            \Vert
                \partial_3 v
            \Vert_{L^2_{x_3}L^4_{x^\prime}(\Omega)}
        \right)\\
        & \leq C \left(
            \Vert
                V
            \Vert_{H^1(\Omega)}
            \Vert
                v
            \Vert_{H^{3/2}(\Omega)}
            + \Vert
                V
            \Vert_{H^{3/2}(\Omega)}
            \Vert
                v
            \Vert_{H^{3/2}(\Omega)}
        \right).
    \end{split}
\end{align}
By the definition of $K_\delta$, we know
\begin{align} \label{eq_estimate_for_mathcal_B_part3}
    \Vert
        \mu P K_\delta V
    \Vert_{L^2(\Omega)}
    \leq C \mu \delta
    \Vert
        V
    \Vert_{H^1(\Omega)}.
\end{align}
Since $\mu + A$ is positive self-adjoint, if $\mu>0$ is so large and $\delta$ is so small that
\begin{align} \label{eq_condition_mu_and_delta_and_norm_of_v_2}
    C (
        \sup_{t>0} \Vert
            v(t)
        \Vert_{H^2(\Omega)}
        + \delta \mu
    )
    \leq \frac{\mu^{3/4}}{2}
\end{align}
for some constant $C>0$, we see that
\begin{align*}
    & \Vert
        \mathcal{B}_{\mu, \delta}(t) V
    \Vert_{L^2(\Omega)}\\
    & \leq C (
        \sup_{t>0} \Vert
            v(t)
        \Vert_{H^2(\Omega)}
        + \delta \mu
    )
    \Vert
        V
    \Vert_{H^{3/2}(\Omega)}\\
    & \leq C (
        \sup_{t>0} \Vert
            v(t)
        \Vert_{H^2(\Omega)}
        + \delta \mu
    )
    \Vert
        A^{3/4} (\mu + A)^{-{3/4}}
    \Vert_{L^2(\Omega) \rightarrow L^2(\Omega)}\\
    & \quad\quad\quad
    \times \Vert
        (\mu + A)^{3/4}V
    \Vert_{L^2(\Omega)}\\
    & \leq \frac{1}{2}
    \Vert
        V
    \Vert_{D((\mu + A)^{3/4})}.
\end{align*}
We find from this estimate and (\ref{eq_lower_bound_for_B_t_varphi_varphi}) and perturbation theory of analytic semigroup that $\mathcal{A}_{\mu, \delta}(t)$ is invertible and a generator of an analytic semigroup.
Moreover the real part of the spectrum of $\mathcal{A}_{\mu, \delta}(t)$ is bounded by $\mu_\ast>0$.
Therefore, we have
\begin{proposition}
    Let $\mu > 1$ and $\delta$ satisfy (\ref{eq_condition_mu_and_delta_and_norm_of_v_1}) and (\ref{eq_condition_mu_and_delta_and_norm_of_v_2}).
    Then the perturbed operator $\mathcal{A}_{\mu, \delta}(t)$ defined by (\ref{eq_definition_of_A_mu_delta}) generates an analytic semigroup on $L^2_{\overline{\sigma}}(\Omega)$ such that
    \begin{align*}
        \Vert
            e^{-t \mathcal{A}_{\mu, \delta} (s)}
        \Vert_{L^2_{\overline{\sigma}} \rightarrow L^2_{\overline{\sigma}}}
        & \leq C e^{- \mu_\ast t},\\
        \Vert
            \mathcal{A}_{\mu, \delta} (s) e^{-t \mathcal{A}_{\mu, \delta} (s)}
        \Vert_{L^2_{\overline{\sigma}} \rightarrow L^2_{\overline{\sigma}}}
        + \Vert
            \partial_t e^{-t \mathcal{A}_{\mu, \delta} (s)}
        \Vert_{L^2_{\overline{\sigma}} \rightarrow L^2_{\overline{\sigma}}}
        & \leq C t^{-1} e^{- \mu_\ast t}.
    \end{align*}
    Moreover, $\mathcal{A}_{\mu, \delta}(s)$ admits its fractional power $\mathcal{A}_{\mu, \delta}(s)^{\alpha}$ such that
    \begin{align*}
        \Vert
            \mathcal{A}_{\mu, \delta} (s)^\alpha e^{-t \mathcal{A}_{\mu, \delta}(s)}
        \Vert_{L^2_{\overline{\sigma}} \rightarrow L^2_{\overline{\sigma}}}
        & \leq C t^{-\alpha} e^{- \mu_\ast t}
    \end{align*}
    for some $C>0$, all $s, t>0$, and $0 \leq \alpha \leq 1$.
\end{proposition}
Let $T_{\mu, \delta} (t,s)$ be the evolution operator associated with $\mathcal{A}_{\mu, \delta}(t)$ constructed later such that $u(t) = T_{\mu, \delta} (t,s) \psi$ for $t>s>0$ is the solution to
\begin{equation} \label{eq_dudt=A_mu_u}
    \begin{aligned}
        \frac{du}{dt}
        & = - \mathcal{A}_{\mu, \delta}(t,s) u(t),
        & 0 < s < t,\\
        u
        & = \psi,
        & t=s.
    \end{aligned}
\end{equation}
The solution to (\ref{eq_diff}) can be represented by the integral formula
\begin{align*}
    V(t)
    & = V(s)
    + \int_s^t
        T_{\mu, \delta}(t, \tau) P
        (
            V(\tau) \cdot \nabla_H V(\tau)
            + W (\tau)\partial_3 V(\tau)
            + F(\tau)
        )
    d\tau.
\end{align*}
We present an abstract lemma for the existence of an evolution operator by Kato and Tanabe \cite{KatoTanabe1962}.
The reader can also refer to the book by Tanabe \cite{Tanabebook1979} for a comprehensive study of evolution operators.
\begin{lemma}[See Theorem 5.2.2-3 in \cite{Tanabebook1979}] \label{lem_existence_of_abstract_evolution_operator}
    Let $T>0$ and $X$ be a Banach space.
    Assume that $\mathcal{A}_{\mu, \delta}(t)$ is a generator of an analytic semigroup such that
    \begin{align}
        \Vert
            (\lambda - \mathcal{A}_{\mu, \delta}(t))^{-1}
        \Vert_{X \rightarrow X}
        \leq \frac{K}{1 + \vert
            \lambda
        \vert}
    \end{align}
    for all $\lambda \in \Sigma_{\theta, \mu}$ for some $\theta \in (0, \pi/2)$ and $\mu>0$.
    Assume that $\mathcal{A}_{\mu, \delta}(t) \mathcal{A}(r)^{-1}$ is a bounded H\"{o}lder continuous operator such that
    \begin{align} \label{eq_assumption_for_bilinear_E_continuity}
        \Vert
            \mathcal{A}(t) \mathcal{A}(r)^{-1}
            - \mathcal{A}(s) \mathcal{A}(r)^{-1}
        \Vert_{X \rightarrow X}
        \leq K \vert
            t - s
        \vert^\alpha
    \end{align}
    holds for all $0 \leq r \leq s \leq t $, $K>0$, $0 < \alpha \leq 1$.
    Assume that $D(\mathcal{Z}(t)) \subset X$ is independent of $t$.
    Then there exists an evolution operator
    \begin{align*}
        S(t,s): X
        \rightarrow D(\mathcal{Z}(t))
        = D(\mathcal{Z}(0))
    \end{align*}
    such that
    \begin{gather} \label{eq_properties_for_T_ts}
        \begin{split}
            \frac{d}{dt} S(t,s)
            = - \mathcal{Z}(t) S(t,s),\\
            \left \Vert
                \frac{d}{dt} S(t,s)
            \right \Vert_{X \rightarrow X}
            + \left \Vert
                \mathcal{Z}(t) S(t,s)
            \right \Vert_{X \rightarrow X}
            \leq C (t - s)^{-1},\\
            \left \Vert
                \frac{d}{ds} S(t,s)
            \right \Vert_{X \rightarrow X}
            \leq C (t - s)^{-1},
        \end{split}
    \end{gather}
    for $0 < s < t < T$.
\end{lemma}
\begin{remark} \label{rmk_for_lem_existence_of_abstract_evolution_operator}
    Assume that $X = L^2_{\overline{\sigma}}$.
    The evolution operator $S(t,s)$ is given by the formula
    \begin{align} \label{eq_formula_for_evolution_operator_S}
        \begin{split}
            S(t,s)
            & = e^{-(t-s)\mathcal{Z}(s)}
            + \int_s^t
                e^{-(t - \tau) \mathcal{Z}(\tau)} R(\tau, s)
            d\tau\\
            & =: e^{-(t-s)\mathcal{Z}(s)}
            + W(t, s),
        \end{split}
    \end{align}
    where $R(t,s)$ is the solution to the integral equation
    \begin{gather} \label{eq_definition_of_R_in_remark}
        \begin{split}
            R(t,s)
            - \int_s^t
                R_1(t,\tau) R(\tau, s)
            d\tau
            = R_1(t,\tau),\\
            R_1(t,s)
            = - (
                \mathcal{Z}(t)
                - \mathcal{Z}(s)
            )e^{-(t - s) \mathcal{Z}(s)},
        \end{split}
    \end{gather}
    for $0 < s < t$.
    Since
    \begin{align*}
        \Vert
            e^{- s \mathcal{Z}(s)}
        \Vert_{L^2_{\overline{\sigma}}(\Omega) \rightarrow L^2_{\overline{\sigma}}(\Omega)}
        & \leq C e^{- \mu_\ast s},\\
        \Vert
            \mathcal{Z}(s) e^{- s \mathcal{Z}(s)}
        \Vert_{L^2_{\overline{\sigma}}(\Omega) \rightarrow L^2_{\overline{\sigma}}(\Omega)}
        & \leq C s^{-1} e^{- \mu_\ast s},\\
        \Vert
            (
                \mathcal{Z}(t)
                - \mathcal{Z}(s)
            )\mathcal{Z}(0)^{-1}
        \Vert_{L^2_{\overline{\sigma}}(\Omega) \rightarrow L^2_{\overline{\sigma}}(\Omega)}
        & \leq C (t - s)^{\alpha},
    \end{align*}
    for some $\mu_\ast>0$ due to $\RealPart \lambda \geq \mu$, we deduce that
    \begin{align*}
        \Vert
            R_1(t, s)
        \Vert_{L^2_{\overline{\sigma}}(\Omega) \rightarrow L^2_{\overline{\sigma}}(\Omega)}
        \leq C_1 (t - s)^{\alpha - 1} e^{- \mu_\ast (t - s)}
    \end{align*}
    for some $t$-independent constant $C>0$.
    For
    \begin{align*}
        R_k (t,s)
        = \int_s^t
            R_1(t,\tau) R_{k-1}(\tau,s)
        d\tau,
    \end{align*}
    by iteration we have
    \begin{align*}
        \Vert
            R_k(t, s)
        \Vert_{L^2_{\overline{\sigma}}(\Omega) \rightarrow L^2_{\overline{\sigma}}(\Omega)}
        \leq (CC_1)^k (t - s)^{k \alpha - 1} \frac{1}{\Gamma(k \alpha)} e^{- \mu_\ast (t - s)}
    \end{align*}
    Since
    \begin{align*}
        R(t, s)
        = \sum_{k=1}^\infty R_k(t,s),
    \end{align*}
    combining with the estimate for $R_k(t,s)$, we see that the series absolutely converge absolutely so that
    \begin{align*}
        \Vert
            R(t, s)
        \Vert_{L^2_{\overline{\sigma}}(\Omega) \rightarrow L^2_{\overline{\sigma}}(\Omega)}
        \leq C (t - s)^{\alpha - 1} e^{- \mu_\ast (t - s)}
    \end{align*}
    where the constant $C>0$ is independent of $t \geq 0$.
    On construction of evolution operator, see the book by Tanabe \cite{Tanabebook1979} and the paper \cite{Tanabe1960} for when the domain is independence of $t$ and \cite{KatoTanabe1962} for when the domain depends on $t$.
    Therefore, we deduce that
    \begin{align*}
        & \Vert
            S(t,s)
        \Vert_{L^2_{\overline{\sigma}}(\Omega) \rightarrow L^2_{\overline{\sigma}}(\Omega)}\\
        & \leq \Vert
            e^{-(t-s)\mathcal{Z}(s)}
        \Vert_{L^2_{\overline{\sigma}}(\Omega) \rightarrow L^2_{\overline{\sigma}}(\Omega)}
        + \left \Vert
            \int_s^t
                e^{-(t - \tau) \mathcal{Z}(\tau)} R(\tau, s)
            d\tau
        \right \Vert_{L^2_{\overline{\sigma}}(\Omega) \rightarrow L^2_{\overline{\sigma}}(\Omega)}\\
        & \leq C e^{- \mu_\ast (t - s)}
        + C \int_s^t
            e^{-\mu_\ast (t - \tau)} (\tau - s)^{\alpha - 1} e^{-\mu_\ast (\tau - s)}
        d\tau\\
        & \leq C e^{- \mu_\ast (t - s)} (
            1
            + (t - s)^{\alpha}
        )\\
        & \leq C e^{- \mu_\ast (t - s)/2}
    \end{align*}
    for some $t$-independent constant $C>0$.
    Moreover, by the formula (\ref{eq_formula_for_evolution_operator_S}) and interpolation, we deduce that
    \begin{align*}
        \Vert
            \mathcal{Z}(t)^{\theta} S(t,s)
        \Vert_{L^2_{\overline{\sigma}}(\Omega) \rightarrow L^2_{\overline{\sigma}}(\Omega)}
        \leq C (t - s)^{- \theta} e^{- \mu_\ast (t - s)/2}
    \end{align*}
    for $\theta \in [0, 1]$.
\end{remark}

\begin{lemma} \label{lem_existence_of_evolution_operator_T_s_t}
    Assume that $v \in C^\theta(0, \infty; H^2(\Omega)^2)$ for $0<\theta<1$.
    Let $\mu$ and $\delta$ satisfy (\ref{eq_condition_mu_and_delta_and_norm_of_v_1}) and (\ref{eq_condition_mu_and_delta_and_norm_of_v_2}).
    Then there exists a solution operator
    \begin{align*}
        T_{\mu, \delta} (t,s): L^2_{\overline{\sigma}}(\Omega)
        \rightarrow D(\mathcal{A}_{\mu, \delta}(t))
        \equiv D(A)
    \end{align*}
    to (\ref{eq_dudt=A_mu_u}) satisfying
    \begin{align*}
        \frac{d}{dt} T_{\mu, \delta} (t,s) V
        = \mathcal{A}_{\mu, \delta}(t) T_{\mu, \delta} (t,s) V, \quad
        t > s > 0, 
    \end{align*}
    and
    \begin{align} \label{eq_estimate_for_T_t_s}
        \begin{split}
            \Vert
                T_{\mu, \delta} (t,s)
            \Vert_{L^2_{\overline{\sigma}}(\Omega) \rightarrow L^2_{\overline{\sigma}}(\Omega)}
            & \leq C e^{- \mu_\ast (t - s)/2},\\
            \Vert
                \partial_t T_{\mu, \delta} (t,s)
            \Vert_{L^2_{\overline{\sigma}}(\Omega) \rightarrow L^2_{\overline{\sigma}}(\Omega)}
            & + \Vert
                \mathcal{A}_{\mu, \delta}(t) T_{\mu, \delta} (t,s)
            \Vert_{L^2_{\overline{\sigma}}(\Omega) \rightarrow L^2_{\overline{\sigma}}(\Omega)}\\
            & \leq C (t - s)^{-1} e^{- \mu_\ast (t - s)/2}
        \end{split}
    \end{align}
    for some constant $C>0$ and $\mu_\ast>0$.
    Moreover, it follows from the interpolation inequality and $t$-independence of $D(\mathcal{A}_{\mu, \delta}(t))$ that
    \begin{align} \label{eq_estimate_for_T_t_s_by_interpolation}
        \Vert
            \mathcal{A}_{\mu, \delta}(t)^{\phi} T_{\mu, \delta} (t,s)
        \Vert_{L^2(\Omega)^2 \rightarrow L^2(\Omega)^2}
        \leq C (t - s)^{-\phi} e^{- \mu_\ast (t - s)/2}
    \end{align}
    for $\phi \in [0,1]$ and some constant $C>0$.
\end{lemma}
\begin{proof}
    We show (\ref{eq_assumption_for_bilinear_E_continuity}).
    It is enough to consider the case that $r=0$.
    Let
    \begin{align*}
        \psi^\prime = \mathcal{A}_{\mu, \delta}(0)^{-1} \varphi \in D(A) \equiv D(\mathcal{A}_{\mu, \delta} (t)) \cap D(\mathcal{A}_{\mu, \delta} (s)), \quad
        \varphi^\prime \in L^2_{\overline{\sigma}}(\Omega).
    \end{align*}
    By definition we observe that
    \begin{align*}
        \mathcal{A}_{\mu, \delta} (t) \psi^\prime
        = (u(t) - u(0)) \cdot \nabla \psi^\prime
        + \psi \cdot (v(t) - v(0))
        + \varphi^\prime.
    \end{align*}
    Therefore, we have
    \begin{align*}
        \mathcal{A}_{\mu, \delta} (t) \psi
        - \mathcal{A}_{\mu, \delta} (s) \psi
        = (u(t) - u(s)) \cdot \nabla \psi^\prime
        + \psi \cdot (v(t) - v(s))
    \end{align*}
    Using (\ref{eq_estimate_for_mathcal_B_part1}) and (\ref{eq_estimate_for_mathcal_B_part2}) we see that
    \begin{align*}
        & \Vert
            \mathcal{A}_{\mu, \delta} (t) \psi
            - \mathcal{A}_{\mu, \delta} (s) \psi
        \Vert_{L^2_{\overline{\sigma}}(\Omega)}\\
        & \leq C (t - s)^\theta \Vert
            \psi
        \Vert_{H^2{(\Omega)}}
        \leq C (t - s)^\theta \Vert
            \varphi
        \Vert_{L^2{(\Omega)}}.
    \end{align*}
    We obtain
    \begin{align} \label{eq_Holder_continuity_of_A_mu_delta_t_over_A_mu_delta_0}
        \Vert
            \mathcal{A}_{\mu, \delta} (t) \mathcal{A}_{\mu, \delta} (0)^{-1}
            - \mathcal{A}_{\mu, \delta} (s) \mathcal{A}_{\mu, \delta} (0)^{-1}
        \Vert_{L^2_{\overline{\sigma}}(\Omega) \rightarrow L^2_{\overline{\sigma}}(\Omega)}
        \leq C (t - s)^\theta.
    \end{align}
    In view of Remark \ref{rmk_for_lem_existence_of_abstract_evolution_operator}, we obtain that
    \begin{align*}
        &\Vert
            T_{\mu, \delta} (t,s)
        \Vert_{L^2_{\overline{\sigma}}(\Omega) \rightarrow L^2_{\overline{\sigma}}(\Omega)}
        \leq C e^{- \mu_\ast (t - s)/2},\\
        & \Vert
            \partial_t T_{\mu, \delta} (t,s)
        \Vert_{L^2_{\overline{\sigma}}(\Omega) \rightarrow L^2_{\overline{\sigma}}(\Omega)}
        + \Vert
            \mathcal{A}_{\mu, \delta}(t) T_{\mu, \delta} (t,s)
        \Vert_{L^2_{\overline{\sigma}}(\Omega) \rightarrow L^2_{\overline{\sigma}}(\Omega)}\\
        & \leq C (t - s)^{-1} e^{- \mu_\ast (t - s)/2},
    \end{align*}
    for some constant $C>0$.
\end{proof}
Thereafter we assume that $\mu$, $\delta$, and $v$ satisfy (\ref{eq_condition_mu_and_delta_and_norm_of_v_1}) and (\ref{eq_condition_mu_and_delta_and_norm_of_v_2}).

\section{Maximal H\"{o}lder regularity of evolution operator and its exponential decay} \label{section_linear_evolution_operator}
\subsection{Supportive estimates}
    We show exponential decay of $T_{\mu, \delta}(t, s)$.
    Although the results in this section are based on Yagi's book \cite{Yagi2010}, the exponential decay is not explicitly addressed in \cite{Yagi2010}.
    In order to prove the main theorem, we clarify this aspect in this section.
    We also focus on the dependence of the basic flow $v$ for estimates of $T_{\mu, \delta}(t, s)$.
    Note that Yagi \cite{Yagi2010} deals with the case where the domain of the operator is time-dependent.
    In our case, the domain is time-independent, and $\mathcal{A}_{\mu, \delta}(t)$ is a perturbation of the hydrostatic Stokes operator, which is a simpler case compared to Yagi's scenario.

    How the constant for bound depends on $\Vert v \Vert_{C^\theta(0, \infty; H^2(\Omega))}$ does not affect the exponential decay itself, but it is useful to know how the deviation of the stability affects $\Vert v \Vert_{C^\theta(0, \infty; H^2(\Omega))}$ when $F_\delta \neq 0$.
    To achieve the maximal H\"{o}lder regularity for $T_{\mu, \delta}(t, s)$, we require the H\"{o}lder regularity to ensure the H\"{o}lder continuity of the resolvent $(\lambda - \mathcal{A}_{\mu, \delta})^{-1}$ in the sense of (\ref{eq_assumption_for_bilinear_E_continuity}) and (\ref{eq_Holder_continuity_of_A_mu_delta_t_over_A_mu_delta_0}).
    In comparison with the $H^1$-convergence, which can be established in the same manner as in \cite{HieberKashiwabara2016, HieberHusseingKashiwabara2016} by replacing the evolution operator $e^{-(t - s) A}$ with $T_{\mu, \delta}(t, s)$, the constants of the $H^2$-bound $C$ in (\ref{eq_decay_estimate_in_main_theorem}) and (\ref{eq_decay_estimate_in_main_theorem2}) worsen by $\Vert v \Vert_{C^\theta(0, \infty; H^2(\Omega))}^2$ as shown in this section.
    Note that the constant of bounds in this section can be taken independent of $T$.

    For a Banach space $X$ and $0 < \sigma < \beta \leq 1$, we define $\mathcal{F}^{\beta, \sigma}(0, T; X)$ as a space of $\sigma$-H\"{o}lder functions such that (i) $\lim_{t \rightarrow 0} t^{1-\beta} \varphi(t)$ makes sense in $X$, (ii) the norm
    \begin{align*}
        & \Vert
            \varphi
        \Vert_{\mathcal{F}^{\beta, \sigma}(0, T; X)}\\
        & := \sup_{0 < t < T} t^{1 - \beta} \Vert
            \varphi(t)
        \Vert_X
        + \sup_{0 < s < t < T} s^{1 - \beta + \sigma}
        \frac{
            \Vert
                \varphi(t)
                - \varphi(s)
            \Vert_X
        }{
            (t - s)^\sigma
        }
    \end{align*}
    is bounded, (iii) $\lim_{t \rightarrow 0} \sup_{0 < s < t } s^{1 - \beta + \sigma} \frac{ \Vert \varphi(t) - \varphi(s) \Vert_X }{(t - s)^\sigma} = 0$.
    We define, in the same manner, the exponentially time-weighted space $\mathcal{F}^{\mu, \beta, \sigma} (0, T; X)$ associated with the norm
    \begin{align*}
        & \Vert
            \varphi
        \Vert_{\mathcal{F}^{\mu, \beta, \sigma}(0, T; X)}\\
        & :=
        \sup_{0 < t < T} e^{\mu } t^{1 - \beta} \Vert
            \varphi(t)
        \Vert_X
        + \sup_{0 < s < t < T} e^{\mu s} s^{1 - \beta + \sigma}
        \frac{
            \Vert
            \varphi(t)
            - \varphi(s)
        \Vert_X
        }{(t - s)^\sigma}.
    \end{align*}
    As in Section 4.8 of the book by Yagi \cite{Yagi2010}, we prove
    \begin{lemma}
        Let $T>0$ and $0 < \sigma < \min (\beta, \theta)$ for $\beta, \theta \in (0,1)$.
        Let
        \begin{align*}
            F
            \in \mathcal{F}^{\beta, \sigma}((0,T]; L^2_{\overline{\sigma}}(\Omega)),\quad
            V_0
            \in D(A^\beta).
        \end{align*}
        Then there exists a solution to
        \begin{align*}
            \frac{dV(t)}{dt} + \mathcal{A}_{\mu, \delta}(t) V(t) = F(t), \quad
            V(0) = V_0,
        \end{align*}
        such that
        \begin{align*}
            & \left \Vert
                \frac{dV(t)}{dt}
            \right \Vert_{L^2(\Omega)}
            + \Vert
                \mathcal{A}_{\mu, \delta}(t) V(t)
            \Vert_{L^2(\Omega)}\\
            & \leq C e^{- \mu_\ast t /4}\left(
                \Vert
                    A^\beta V_0
                \Vert_{X}
                + \Vert
                    F
                \Vert_{\mathcal{F}^{\beta, \sigma}((0,T]; L^2(\Omega))}
            \right).
        \end{align*}
        where $C > 0$ can be taken $t$-independently.
    \end{lemma}
    We use this theorem for $\beta = 1/2$, namely $V_0 \in H^1(\Omega)^2$.
    \subsubsection{Estimates for the semigroup}
    \begin{proposition} \label{prop_diff_estimate_with_resolvlents_at_t_s}
        Let $t > s > 0$ and $\lambda \in \Complex \setminus \overline{\Sigma_{\phi, \mu_\ast}}$ for the spectral angle $\phi \in (0, \pi/2)$.
        Then
        \begin{align*}
            \Vert
                (\lambda - \mathcal{A}_{\mu, \delta}(t))^{-1}
                - (\lambda - \mathcal{A}_{\mu, \delta}(s))^{-1}
            \Vert_{L^2_{\overline{\sigma}} \rightarrow L^2_{\overline{\sigma}}}
            \leq \frac{
                C \min (
                    1,
                    \vert
                        t - s
                    \vert^{\theta}
                )
            }{
                \vert
                    \lambda
                \vert
            }.
        \end{align*}
    \end{proposition}
    \begin{proof}
        It is deduced from the resolvent estimate, the estimate (\ref{eq_Holder_continuity_of_A_mu_delta_t_over_A_mu_delta_0}), and the formula
        \begin{align*}
            & (\lambda - \mathcal{A}_{\mu, \delta}(t))^{-1}
            - (\lambda - \mathcal{A}_{\mu, \delta}(s))^{-1}\\
            & = (
                \lambda - \mathcal{A}_{\mu, \delta}(t)
            )^{-1}
            D_{\mu, \delta}(t, s) \mathcal{A}_{\mu, \delta}(s)
            (
                \lambda - \mathcal{A}_{\mu, \delta}(s)
            )^{-1},
        \end{align*}
        where
        \begin{align*}
            D_{\mu, \delta}(t,s)
            & := - (\mathcal{A}_{\mu, \delta}(t) - \mathcal{A}_{\mu, \delta}(s)) \mathcal{A}_{\mu, \delta}(s)^{-1}\\
            & = \mathcal{A}_{\mu, \delta}(t) (\mathcal{A}_{\mu, \delta}(t)^{-1} - \mathcal{A}_{\mu, \delta}(s)^{-1}).
        \end{align*}
    \end{proof}

    \begin{proposition}\label{prop_estimate_for_At_phi_As_phi_inverse}
        Let $0 < \phi < 1$ and $0 < s, t < \infty$.
        Then
        \begin{align} \label{eq_estimate_for_At_phi_As_phi_inverse}
            \Vert
                \mathcal{A}_{\mu, \delta}(t)^\phi \mathcal{A}_{\mu, \delta}(s)^{-\phi}
            \Vert_{L^2_{\overline{\sigma}} \rightarrow L^2_{\overline{\sigma}}}
            \leq C,
        \end{align}
        where
        \begin{align*}
            C = O(
                \Vert
                    v
                \Vert_{L^\infty(0,T; H^2(\Omega)^2)}
            ).
        \end{align*}
    \end{proposition}
    \begin{proof}
        Let $x \in D(\mathcal{A}_{\mu, \delta}(t)) \equiv D(\mathcal{A}(0))$.
        We first show that
        \begin{align}\label{eq_At}
            \Vert
                (
                    \mathcal{A}_{\mu, \delta}(s)^\phi - \mathcal{A}_{\mu, \delta}(t)^\phi
                ) \mathcal{A}_{\mu, \delta}(s)^{-\phi}
            \Vert_{L^2_{\overline{\sigma}}(\Omega) \rightarrow L^2_{\overline{\sigma}}(\Omega)}
            \leq C \min(1, (t - s)^\theta)
        \end{align}
        for $C = O(\Vert v \Vert_{L^\infty(0,T; H^2(\Omega)^2)})$.
        We invoke the formula
        \begin{align*}
            \mathcal{A}_{\mu, \delta}(t)^\phi x
            = - \frac{\sin (\pi \phi)}{\pi} \int_0^\infty
                \rho^{\phi - 1} (\rho + \mathcal{A}_{\mu, \delta}(t))^{-1} \mathcal{A}_{\mu, \delta}(t) x
            d\rho.
        \end{align*}
        Therefore
        \begin{align*}
            & (
                \mathcal{A}_{\mu, \delta}(s)^\phi - \mathcal{A}_{\mu, \delta}(t)^\phi
            ) \mathcal{A}_{\mu, \delta}(s)^{-\phi} x\\
            & = - \frac{\sin (\pi \phi)}{\pi} \int_0^\infty
                \rho^{\phi} (
                    \rho + \mathcal{A}_{\mu, \delta}(t)
                )^{-1} (
                    \mathcal{A}_{\mu, \delta}(s) - \mathcal{A}_{\mu, \delta}(t)
                )\\
            & \quad\quad\quad\quad\quad\quad\quad\quad\quad
                \times (
                    \rho + \mathcal{A}_{\mu, \delta}(s)
                )^{-1} \mathcal{A}_{\mu, \delta}(s)^{-\phi}
            d\rho.
        \end{align*}
        By definition of $\mathcal{A}_{\mu, \delta}(t)$, we deduce that
        \begin{align*}
            & \Vert
                (\mathcal{A}_{\mu, \delta}(t) - \mathcal{A}_{\mu, \delta}(s)) y
            \Vert_{L^2_{\overline{\sigma}}}\\
            & \leq C \Vert
                u
            \Vert_{L^\infty(0,T; H^2(\Omega)^2)}
            \min(1, (t - s)^\theta)
            \Vert
                \mathcal{A}_{\mu, \delta}(s)^{3/4} y
            \Vert_{L^2_{\overline{\sigma}}}
        \end{align*}
        for $y \in D(\mathcal{A}_{\mu, \delta}(s)^{3/4})$.
        Therefore, we find from Proposition \ref{prop_diff_estimate_with_resolvlents_at_t_s} that
        \begin{align} \label{eq_estimate_for_At_phi_minus_As_phi_As_phi_inverse}
            \begin{split}
                & \biggl \Vert
                    \int_0^\infty
                        \rho^{\phi} (
                            \rho + \mathcal{A}_{\mu, \delta}(t)
                        )^{-1} (
                            \mathcal{A}_{\mu, \delta}(s) - \mathcal{A}_{\mu, \delta}(t)
                        )\\
                & \quad\quad\quad\quad\quad\quad\quad\quad
                        \times (
                            \rho + \mathcal{A}_{\mu, \delta}(s)
                        )^{-1} \mathcal{A}_{\mu, \delta}(s)^{-\phi}
                    d\rho
                \biggr \Vert_{L^2_{\overline{\sigma}}}\\
                & = \biggl \Vert
                    \int_0^\infty
                        \rho^{\phi} (
                            \rho + \mathcal{A}_{\mu, \delta}(t)
                        )^{-1} (
                            \mathcal{A}_{\mu, \delta}(s) - \mathcal{A}_{\mu, \delta}(t)
                        ) \mathcal{A}_{\mu, \delta}(s)^{-3/4}\\
                    & \quad\quad\quad\quad\quad\quad\quad\quad\quad\quad\quad\quad
                            \times\mathcal{A}_{\mu, \delta}(s)^{3/4 - \phi} (
                            \rho + \mathcal{A}_{\mu, \delta}(s)
                        )^{-1}
                    d\rho
                \biggr \Vert_{L^2_{\overline{\sigma}}}\\
                & \leq C \min(1, (t - s)^\theta) \\
                &\quad\quad \times \int_0^\infty
                    \max(\rho^{\phi} (\mu_\ast + \rho)^{-5/4 -\phi}, \rho^{\phi} (\mu_\ast + \rho)^{-2})
                d\rho \Vert
                    x
                \Vert_{L^2_{\overline{\sigma}}}\\
                & \leq C \min(1, (t - s)^\theta) \Vert
                    x
                \Vert_{L^2_{\overline{\sigma}}},
            \end{split}
        \end{align}
        where $C = O(\Vert v \Vert_{L^\infty(0,T; H^2(\Omega)^2)})$.
        Using density argument and the formula
        \begin{align*}
            \mathcal{A}_{\mu, \delta}(t)^{\phi} \mathcal{A}_{\mu, \delta}(s)^{-\phi}
            = I - (\mathcal{A}_{\mu, \delta}(s)^\phi - \mathcal{A}_{\mu, \delta}(t)^\phi)\mathcal{A}_{\mu, \delta}(s)^{-\phi},
        \end{align*}
        we obtain (\ref{eq_estimate_for_At_phi_As_phi_inverse}).
    \end{proof}

    \begin{proposition} \label{prop_eq_estimate_for_At_phi_minus_As_phi_As_phi_inverse_etsAs}
        Let $0 < \phi < 1$.
        Then
        \begin{align*}
            & \Vert
                \mathcal{A}_{\mu, \delta}(t)^\phi e^{- \tau \mathcal{A}_{\mu, \delta}(s)} \mathcal{A}_{\mu, \delta}(s)^{-\phi}
                - e^{-\tau \mathcal{A}_{\mu, \delta}(s)}
            \Vert_{L^2_{\overline{\sigma}} \rightarrow L^2_{\overline{\sigma}}}\\
            & \leq C \min(1, (t - s)^\theta) e^{- \mu_\ast (t - s)},
        \end{align*}
    \end{proposition}
    where $C = O(\Vert v \Vert_{L^\infty(0,T; H^2(\Omega)^2)})$.
    \begin{proof}
        This is a direct consequence from the identity
        \begin{align*}
            e^{- \tau \mathcal{A}_{\mu, \delta}(s)}
            = \mathcal{A}_{\mu, \delta}(s)^\phi e^{- \tau \mathcal{A}_{\mu, \delta}(s)} \mathcal{A}_{\mu, \delta}(s)^{- \phi}
        \end{align*}
        and the estimate (\ref{eq_estimate_for_At_phi_minus_As_phi_As_phi_inverse}).
    \end{proof}

    In view of Proposition \ref{prop_estimate_for_At_phi_As_phi_inverse}, we can also deduce
    \begin{proposition} \label{prop_estimate_At_phi_e_minus_t_s_As}
        Let $0 \leq \phi, \psi \leq 1$ and $0 < s, t < \infty$.
        Then
        \begin{align*}
            & \Vert
                \mathcal{A}_{\mu, \delta}(t)^\phi e^{- \tau \mathcal{A}_{\mu, \delta}(s)}
            \Vert_{L^2_{\overline{\sigma}} \rightarrow L^2_{\overline{\sigma}}}
            \leq C (
                1
                + \vert
                    t - s
                \vert^\theta
            )
            \tau^{- \phi} e^{- \mu _\ast \tau},
        \end{align*}
        and
        \begin{align*}
            & \Vert
                \mathcal{A}_{\mu, \delta}(t)^\phi e^{- \tau \mathcal{A}_{\mu, \delta}(s)} \mathcal{A}_{\mu, \delta}(s)^{- \psi}
            \Vert_{L^2_{\overline{\sigma}} \rightarrow L^2_{\overline{\sigma}}}\\
            & + \Vert
                \mathcal{A}_{\mu, \delta}(t)^\phi e^{- \tau \mathcal{A}_{\mu, \delta}(t)} \mathcal{A}_{\mu, \delta}(s)^{- \psi}
            \Vert_{L^2_{\overline{\sigma}} \rightarrow L^2_{\overline{\sigma}}}\\
            & \leq C (
                1
                + \vert
                    t - s
                \vert^\theta
            )
            \tau^{- \phi + \psi} e^{- \mu _\ast \tau},
        \end{align*}
        where $C = O(\Vert v \Vert_{L^\infty(0,T; H^2(\Omega)^2)})$.
    \end{proposition}
    \begin{proof}
        We observe that
        \begin{align*}
            & \mathcal{A}_{\mu, \delta}(t)^\phi e^{- \tau \mathcal{A}_{\mu, \delta}(s)}\\
            & = \mathcal{A}_{\mu, \delta}(t)^\phi e^{- \tau \mathcal{A}_{\mu, \delta}(t)}
            - \mathcal{A}_{\mu, \delta}(t)^\phi(
                e^{- \tau \mathcal{A}_{\mu, \delta}(t)}
                - e^{- \tau \mathcal{A}_{\mu, \delta}(s)}
            ).
        \end{align*}
        It is clear that
        \begin{align*}
            \Vert
                \mathcal{A}_{\mu, \delta}(t)^\phi e^{- \tau \mathcal{A}_{\mu, \delta}(t)}
            \Vert_{L^2_{\overline{\sigma}} \rightarrow L^2_{\overline{\sigma}}}
            \leq C  \tau^{- \phi} e^{- \mu _\ast  \tau}.
        \end{align*}
        We estimate the second term.
        Since
        \begin{align*}
            & \mathcal{A}_{\mu, \delta}(t)^\phi
            (e^{- \tau \mathcal{A}_{\mu, \delta}(s)}
                - e^{- \tau \mathcal{A}_{\mu, \delta}(s)}
            )\\
            & = \frac{1}{2 \pi i} \int_{\Gamma}
                e^{-\lambda  \tau} \mathcal{A}_{\mu, \delta}(t)^\phi (\lambda - \mathcal{A}_{\mu, \delta}(t))^{-1}\\
            & \quad\quad\quad\quad\quad\quad \times
                D_{\mu, \delta}(t,s) \mathcal{A}_{\mu, \delta}(s) (\lambda - \mathcal{A}_{\mu, \delta}(s))^{-1}
            d \lambda,
        \end{align*}
        we deduce that
        \begin{align*}
            & \Vert
                \mathcal{A}_{\mu, \delta}(t)^\phi(
                    e^{- \tau \mathcal{A}_{\mu, \delta}(s)}
                    - e^{- \tau \mathcal{A}_{\mu, \delta}(s)}
                )
            \Vert_{\Vert_{L^2_{\overline{\sigma}} \rightarrow L^2_{\overline{\sigma}}}}\\
            & \leq C \int_\Gamma
                e^{
                    - \vert
                        \lambda
                    \vert \tau
                    - \mu_\ast \tau
                } \min(1, (\mu_\ast + r)^{\phi - 1}) \tau^\theta
            d \vert \lambda \vert\\
            & \leq C \tau^{- \phi} \min (1, (t - s)^\theta) e^{- \mu _\ast \tau}
        \end{align*}
        for some $C = O(\Vert v \Vert_{L^\infty(0,T; H^2(\Omega)^2)})$.
        We have obtained the first inequality.
        Since
        \begin{align*}
            \mathcal{A}_{\mu, \delta}(t)^\phi e^{- \tau \mathcal{A}_{\mu, \delta}(s)} \mathcal{A}_{\mu, \delta}(s)^{- \psi}
            & = \mathcal{A}_{\mu, \delta}(t)^\phi \mathcal{A}_{\mu, \delta}(s)^{- \phi} \mathcal{A}_{\mu, \delta}(t)^{\phi - \psi} e^{- \tau \mathcal{A}_{\mu, \delta}(s)},\\
            \mathcal{A}_{\mu, \delta}(t)^\phi e^{- \tau \mathcal{A}_{\mu, \delta}(t)} \mathcal{A}_{\mu, \delta}(s)^{- \psi}
            & = \mathcal{A}_{\mu, \delta}(t)^{\phi - \psi} e^{- \tau \mathcal{A}_{\mu, \delta}(t)} \mathcal{A}_{\mu, \delta}(t)^{- \psi} \mathcal{A}_{\mu, \delta}(t)^\psi,
        \end{align*}
        combining this formula with Proposition \ref{prop_estimate_for_At_phi_As_phi_inverse}, we obtain the second inequality.
    \end{proof}

    \begin{proposition} \label{prop_continuity_e_tau_At}
        Let $\tau > 0$ and $0 < s \leq t < \infty$.
        Let $0 \leq \phi, \psi \leq 1$.
        Then
        \begin{align*}
            & \Vert
                \mathcal{A}_{\mu, \delta}(t)^\phi (
                    e^{- \tau \mathcal{A}_{\mu, \delta}(t) }
                    - e^{- \tau \mathcal{A}_{\mu, \delta}(s) }
                ) \mathcal{A}_{\mu, \delta}(s)^{- \psi}
            \Vert_{L^2_{\overline{\sigma}}(\Omega) \rightarrow L^2_{\overline{\sigma}}(\Omega)}\\
            &\leq C \tau^{- \phi + \psi} \vert
                t - s
            \vert^\theta e^{-\mu_\ast \tau},
        \end{align*}
        where $C = O(\Vert v \Vert_{L^\infty(0,T; H^2(\Omega)^2)})$.
    \end{proposition}
    \begin{proof}
        See the proof of (4.91) in Yagi \cite{Yagi2010}.
    \end{proof}

    \begin{proposition} \label{prop_continuity_At_phi_e_tau_At}
        Let $\tau > 0$ and $0 < s \leq t < \infty$.
        Let $0 \leq \phi \leq 2$ and $0 \leq \psi \leq 1$.
        Then
        \begin{align} \label{eq_continuity_At_phi_e_tau_At}
            \begin{split}
                & \Vert
                    (
                        \mathcal{A}_{\mu, \delta}(t)^\phi e^{- \tau \mathcal{A}_{\mu, \delta}(t) }
                        - \mathcal{A}_{\mu, \delta}(s)^\phi e^{- \tau \mathcal{A}_{\mu, \delta}(s) }
                    ) \mathcal{A}_{\mu, \delta}(s)^{- \psi}
                \Vert_{L^2_{\overline{\sigma}}(\Omega) \rightarrow L^2_{\overline{\sigma}}(\Omega)}\\
                & \leq C \tau^{- \phi + \psi} \vert
                    t - s
                \vert^\theta e^{-\mu_\ast \tau},
            \end{split}
        \end{align}
        where $C = O(\Vert v \Vert_{L^\infty(0,T; H^2(\Omega)^2)})$.
    \end{proposition}
    \begin{proof}
        Since
        \begin{align*}
            & ( \mathcal{A}_{\mu, \delta}(t)^\phi e^{-(t - s) \mathcal{A}_{\mu, \delta}(s)}
                - \mathcal{A}_{\mu, \delta}(s)^\phi e^{-(t - s) \mathcal{A}_{\mu, \delta}(s)}
            ) \mathcal{A}_{\mu, \delta}(s)^{- \psi}\\
            & = \frac{1}{2 \pi i} \int_{\Gamma}
                \lambda^\phi e^{-\lambda (t - s)} (
                    \lambda - \mathcal{A}_{\mu, \delta}(t)
                )^{-1}\\
            & \quad\quad\quad\quad\quad \times
                D_{\mu, \delta}(t,s) \mathcal{A}_{\mu, \delta}(s)^{1 - \psi} (
                    \lambda - \mathcal{A}_{\mu, \delta}(s)
                )^{-1}
            d \lambda,
        \end{align*}
        similar to the proposition \ref{prop_estimate_At_phi_e_minus_t_s_As}, we deduce (\ref{eq_continuity_At_phi_e_tau_At}).
    \end{proof}

    \begin{proposition} \label{prop_estimate_I_minus_e_minus_t_s_with_integral}
        There exists a $t$- and $s$-independent constant $C>0$ such that
        \begin{align*}
            & \left \Vert
                (
                    I
                    - e^{-(t - s) \mathcal{A}_{\mu, \delta}(t)}
                )
                \int_0^t
                    \mathcal{A}_{\mu, \delta}(t) e^{- (t - \tau) \mathcal{A}_{\mu, \delta}(t)}
                    (
                        f(t)
                        - f(\tau)
                    )
                d\tau
            \right \Vert\\
            & \leq C e^{- \eta s} (t - s)^\sigma s^{- 1 + \beta - \sigma}
        \end{align*}
        for all vector field $f \in \mathcal{F}^{\eta, \sigma, \beta}(0, T; L^2_{\overline{\sigma}}(\Omega))$ and $\eta \in [0, \mu_\ast/2]$, where
        \begin{align*}
            C
            = O(\Vert v \Vert_{L^\infty(0,T; H^2(\Omega)^2)}).
        \end{align*}
    \end{proposition}
    \begin{proof}
        See (4.21) in the proof of Theorem 4.6 in \cite{Yagi2010}.
    \end{proof}

    \subsubsection{Estimates for the evolution operator}
    Let $R_{\mu, \delta}(t,s)$ be the bounded operator defined by (\ref{eq_definition_of_R_in_remark}) with $\mathcal{Z}(t) = \mathcal{A}_{\mu, \delta}(t)$.

    \begin{proposition}\label{prop_estimate_for_At_phi_Tts}
        Let $0 < \phi < 1$ and $0 \leq s \leq t < \infty$.
        Then
        \begin{align}\label{eq_estimate_for_At_phi_Tts}
            \begin{split}
                \Vert
                    \mathcal{A}_{\mu, \delta}(t)^\phi T_{\mu, \delta}(t,s)
                \Vert_{L^2_{\overline{\sigma}} \rightarrow L^2_{\overline{\sigma}}}
                \leq C (t - s)^{- \phi} e^{- \mu _\ast (t - s)/2},
            \end{split}
        \end{align}
        where $C = O(\Vert v \Vert_{L^\infty(0,T; H^2(\Omega)^2)})$.
    \end{proposition}
    \begin{proof}
        We invoke
        \begin{align} \label{eq_formula_for_Tts_appendices}
            \begin{split}
                T_{\mu, \delta}(t,s)
                & = e^{-(t - s)\mathcal{A}_{\mu, \delta}(s)}
                + \int_s^t
                    e^{-(t - \tau)\mathcal{A}(\tau)} R_{\mu, \delta}(\tau,s)
                d\tau\\
                & =: e^{-(t - s)\mathcal{A}_{\mu, \delta}(s)}
                + W_{\mu, \delta}(t,s).
            \end{split}
        \end{align}
        Since
        \begin{align*}
            & \Vert
                \mathcal{A}_{\mu, \delta}(t)^\phi e^{-(t - \tau) \mathcal{A}(\tau)} R_{\mu, \delta}(\tau, s)
            \Vert_{L^2_{\overline{\sigma}}(\Omega) \rightarrow L^2_{\overline{\sigma}}(\Omega)}\\
            & = \Vert
                \mathcal{A}_{\mu, \delta}(t)^\phi \mathcal{A}_{\mu, \delta}(\tau)^{-\phi} \mathcal{A}(\tau)^\phi e^{-(t - \tau) \mathcal{A}_{\mu, \delta}(\tau)} R_{\mu, \delta}(\tau, s)
            \Vert_{L^2_{\overline{\sigma}}(\Omega) \rightarrow L^2_{\overline{\sigma}}(\Omega)}\\
            & \leq C (t - \tau)^{-\phi} (\tau - s)^{-1 + \theta} e^{-\mu_\ast (t - s)},
        \end{align*}
        for $C = O(\Vert v \Vert_{L^\infty(0,T; H^2(\Omega)^2)})$, we find that
        \begin{align} \label{eq_estimate_for_At_phi_Wts}
            \begin{split}
                & \Vert
                    \mathcal{A}_{\mu, \delta}(t)^\phi W_{\mu, \delta}(t,s)
                \Vert_{L^2_{\overline{\sigma}} \rightarrow L^2_{\overline{\sigma}}}\\
                & \leq C e^{- \mu _\ast (t - s)}\int_s^t
                        (t - \tau)^{-\phi} (\tau - s)^{-1 + \theta}
                    d \tau\\
                & \leq C e^{- \mu _\ast (t - s)} (t - s)^{-\phi + \theta} B(1 - \phi, \theta),
            \end{split}
        \end{align}
        where $B(a,b)$ for $a,b>0$ is the beta-function.
        Therefore, we obtain
        \begin{align*}
            \Vert
                \mathcal{A}_{\mu, \delta}(t)^\phi T_{\mu, \delta}(t,s)
            \Vert_{L^2_{\overline{\sigma}} \rightarrow L^2_{\overline{\sigma}}}
            \leq C e^{- \mu _\ast (t - s)/2} (t - s)^{- \phi}
        \end{align*}
        for $C = O(\Vert v \Vert_{L^\infty(0,T; H^2(\Omega)^2)})$.
    \end{proof}
    Similar to Proposition \ref{prop_estimate_for_At_phi_Tts}, we deduce
    \begin{proposition}\label{prop_estimate_for_At_phi_Tts_As_phi_phi_inverse}
        Let $0 < \psi \leq \phi < 1$ and $0 \leq s \leq t < \infty$.
        Then
        \begin{align*}
            & \Vert
                \mathcal{A}_{\mu, \delta}(t)^\phi (
                    T_{\mu, \delta}(t,s)
                    - e^{- (t - s) \mathcal{A}_{\mu, \delta}(s)}
                ) \mathcal{A}_{\mu, \delta}(s)^{- \psi}
            \Vert_{L^2_{\overline{\sigma}} \rightarrow L^2_{\overline{\sigma}}}\\
            & \leq C (t - s)^{- \phi + \psi + \theta} e^{- \mu _\ast (t - s)},
        \end{align*}
        and
        \begin{align*}
            & \Vert
                \mathcal{A}_{\mu, \delta}(t)^\phi T_{\mu, \delta}(t,s) \mathcal{A}_{\mu, \delta}(s)^{- \psi}
            \Vert_{L^2_{\overline{\sigma}} \rightarrow L^2_{\overline{\sigma}}}\\
            & \leq C (t - s)^{- \phi + \psi} e^{- \mu _\ast (t - s)/2},
        \end{align*}
        where $C = O(\Vert v \Vert_{L^\infty(0,T; H^2(\Omega)^2)})$.
    \end{proposition}
    \begin{proof}
        Similar to the proof of Proposition \ref{prop_estimate_for_At_phi_Tts}, we deduce that
        \begin{align} \label{eq_estimate_for_At_psi_remainder_term_As_phi_minus_psi}
            \begin{split}
                \left \Vert
                    \mathcal{A}_{\mu, \delta}(t)^\phi
                    W(t,s)
                    \mathcal{A}_{\mu, \delta}(s)^{- \psi}
                \right \Vert_{L^2_{\overline{\sigma}} \rightarrow L^2_{\overline{\sigma}}}
                & \leq C e^{- \mu _\ast (t - s)} (t - s)^{- \phi + \psi + \theta}\\
                & \leq C e^{- \mu _\ast (t - s)/2} (t - s)^{- \phi + \psi}.
            \end{split}
        \end{align}
        We obtain the first estimate.
        Proposition \ref{prop_estimate_for_At_phi_As_phi_inverse} leads to the same estimate as (\ref{prop_estimate_for_At_phi_Tts_As_phi_phi_inverse}) for $\mathcal{A}_{\mu, \delta}(t)^\phi e^{-(t - s)\mathcal{A}_{\mu, \delta}(s)} \mathcal{A}_{\mu, \delta}(s)^{- \psi}$.
    \end{proof}

    \begin{proposition} \label{eq_estimate_At_phi_Tts_minus_e_minus_t_s_At}
        Let $0 \leq \phi < 1 + \theta$ and $0 \leq \psi \leq 1$.
        Then
        \begin{align*}
            &\Vert
                \mathcal{A}_{\mu, \delta}(t)^\phi
                (
                    T_{\mu, \delta}(t, s)
                    - e^{-(t - s)\mathcal{A}_{\mu, \delta}(t)}
                )
                \mathcal{A}_{\mu, \delta}(s)^{-\psi}
            \Vert_{L^2_{\overline{\sigma}} \rightarrow L^2_{\overline{\sigma}}}\\
            & \leq C \min \left(
                (t - s)^{-\phi + \psi + \theta} e^{- \mu_\ast (t - s)},
                (t - s)^{-\phi + \psi} e^{- \mu_\ast (t - s)/2}
            \right),
        \end{align*}
        where $C = O(\Vert v \Vert_{L^\infty(0,T; H^2(\Omega)^2)})$.
    \end{proposition}
    \begin{proof}
        Propositions \ref{prop_continuity_e_tau_At} and \ref{prop_estimate_for_At_phi_Tts_As_phi_phi_inverse} imply the estimate.
    \end{proof}

    Propositions \ref{prop_estimate_At_phi_e_minus_t_s_As} and \ref{eq_estimate_At_phi_Tts_minus_e_minus_t_s_At} imply
    \begin{corollary} \label{eq_estimate_At_phi_Tts_At}
        Let $0 \leq \phi < 1 + \theta$ and $0 \leq \psi \leq 1$.
        Then
        \begin{align*}
            &\Vert
                \mathcal{A}_{\mu, \delta}(t)^\phi
                    T_{\mu, \delta}(t, s)
                \mathcal{A}_{\mu, \delta}(s)^{-\psi}
            \Vert_{L^2_{\overline{\sigma}} \rightarrow L^2_{\overline{\sigma}}}\\
            & \leq C \min \left(
                (t - s)^{-\phi + \psi + \theta} e^{- \mu_\ast (t - s)},
                (t - s)^{-\phi + \psi} e^{- \mu_\ast (t - s)/2}
            \right).
        \end{align*}
        where $C = O(\Vert v \Vert_{L^\infty(0,T; H^2(\Omega)^2)})$.
    \end{corollary}

    \begin{proposition} \label{prop_At_phi_Tts_As_minus_phi_minus_e_t_s_As}
        Let $0 < \psi \leq \phi < 1$.
        Then
        \begin{align*}
            & \Vert
                \mathcal{A}_{\mu, \delta}(t)^\phi T_{\mu, \delta}(t, s) \mathcal{A}_{\mu, \delta}(s)^{-\phi}
                - e^{-(t - s)\mathcal{A}_{\mu, \delta}(s)}
            \Vert_{L^2_{\overline{\sigma}} \rightarrow L^2_{\overline{\sigma}}}\\
            & \leq C (t - s)^{\theta} e^{- \mu_\ast (t - s)},
        \end{align*}
        where $C = O(\Vert v \Vert_{L^\infty(0,T; H^2(\Omega)^2)})$.
    \end{proposition}
    \begin{proof}
        Combining the formula
        \begin{align*}
            & \mathcal{A}_{\mu, \delta}(t)^\phi T_{\mu, \delta}(t, s) \mathcal{A}_{\mu, \delta}(s)^{-\phi}
            - e^{-(t - s)\mathcal{A}_{\mu, \delta}(s)}\\
            & =
            \mathcal{A}_{\mu, \delta}(t)^\phi (
                T_{\mu, \delta}(t, s)
                - e^{-(t - s)\mathcal{A}_{\mu, \delta}(t)}
            ) \mathcal{A}_{\mu, \delta}(s)^{-\phi}\\
            & +
            (
                \mathcal{A}_{\mu, \delta}(t)^\phi e^{-(t - s)\mathcal{A}_{\mu, \delta}(t)}
                - \mathcal{A}_{\mu, \delta}(s)^\phi e^{-(t - s)\mathcal{A}_{\mu, \delta}(s)}
            ) \mathcal{A}_{\mu, \delta}(s)^{-\phi}
        \end{align*}
        with Propositions \ref{prop_continuity_At_phi_e_tau_At} and \ref{eq_estimate_At_phi_Tts_minus_e_minus_t_s_At}, we obtain the estimate.
    \end{proof}
    We define
    \begin{align*}
        \tilde{W}_{\mu, \delta}(t, s)
        = T_{\mu, \delta}(t, s)
        - e^{- (t - \tau) \mathcal{A}_{\mu, \delta}(t)}.
    \end{align*}

    \begin{proposition} \label{prop_estimate_int_t_tildeWttau_minus_Tts_tildeWstau}
        Let $0 < s < t$.
        Then
        \begin{align*}
            & \left \Vert
                \int_0^s
                    \mathcal{A}_{\mu, \delta}(t) (
                        \tilde{W}_{\mu, \delta}(t, \tau)
                        - T_{\mu, \delta}(t, s) \tilde{W}_{\mu, \delta}(s, \tau)
                    )
                d\tau
            \right \Vert_{L^2(\Omega)^2 \rightarrow L^2(\Omega)^2}\\
            & \leq C \vert
                t - s
            \vert^{\theta},
        \end{align*}
        where $C = O(\Vert v \Vert_{L^\infty(0,T; H^2(\Omega)^2)})$.
    \end{proposition}
    \begin{proof}
        We observe that
        \begin{align*}
            & \mathcal{A}_{\mu, \delta}(t) \left(
                \tilde{W}_{\mu, \delta}(t, \tau)
                - T_{\mu, \delta}(t, s) \tilde{W}_{\mu, \delta}(s, \tau)
            \right)\\
            & = - \mathcal{A}_{\mu, \delta}(t) e^{- (t - \tau) \mathcal{A}_{\mu, \delta}(t)}    
            + \mathcal{A}_{\mu, \delta}(t) T_{\mu, \delta}(t, s) e^{- (s - \tau) \mathcal{A}_{\mu, \delta}(s)}\\
            & = \mathcal{A}_{\mu, \delta}(s) e^{- (t - \tau) \mathcal{A}_{\mu, \delta}(s)}
                - \mathcal{A}_{\mu, \delta}(t) e^{- (t - \tau) \mathcal{A}_{\mu, \delta}(t)}\\
            & + (
                \mathcal{A}_{\mu, \delta}(t) T_{\mu, \delta}(t, s) \mathcal{A}_{\mu, \delta}(s)^{-1}
                - e^{- (t - s) \mathcal{A}_{\mu, \delta}(s)}
            ) \mathcal{A}_{\mu, \delta}(s) e^{- (s - \tau) \mathcal{A}_{\mu, \delta}(s)}.
        \end{align*}
        Therefore by Proposition \ref{prop_diff_estimate_with_resolvlents_at_t_s} and definition of the analytic semigroup imply that
        \begin{align*}
            & \left \Vert
                \int_0^s
                    \mathcal{A}_{\mu, \delta}(s) e^{- (t - \tau) \mathcal{A}_{\mu, \delta}(s)}
                    - \mathcal{A}_{\mu, \delta}(t) e^{- (t - \tau) \mathcal{A}_{\mu, \delta}(t)}
                d\tau
            \right \Vert_{L^2(\Omega)^2 \rightarrow L^2(\Omega)^2}\\
            & = \bigl \Vert
                (
                    e^{- (t - s) \mathcal{A}_{\mu, \delta}(s)}
                    - e^{- t \mathcal{A}_{\mu, \delta}(s)}
                )\\
            & \quad\quad\quad\quad
                - (
                    e^{- (t - s) \mathcal{A}_{\mu, \delta}(t)}
                    - e^{- t \mathcal{A}_{\mu, \delta}(t)}
                )
            \bigr \Vert_{L^2(\Omega)^2 \rightarrow L^2(\Omega)^2}\\
            & \leq C(t - s)^\theta e^{- \mu_\ast (t - s)}
        \end{align*}
        and
        \begin{align*}
            & \biggl \Vert
                \int_0^s
                (
                    \mathcal{A}_{\mu, \delta}(t) T_{\mu, \delta}(t, s) \mathcal{A}_{\mu, \delta}(s)^{-1}
                    - e^{- (t - s) \mathcal{A}_{\mu, \delta}(s)}
                )\\
            & \quad\quad\quad\quad\quad\quad\quad\quad\quad\quad \times\mathcal{A}_{\mu, \delta}(s) e^{- (s - \tau) \mathcal{A}_{\mu, \delta}(s)}
                d\tau
            \biggr \Vert_{L^2(\Omega)^2 \rightarrow L^2(\Omega)^2}\\
            & \leq C(t - s)^\theta
            \Vert
                I
                - e^{- (t - s) \mathcal{A}_{\mu, \delta}(s)}
            \Vert_{L^2(\Omega)^2 \rightarrow L^2(\Omega)^2}\\
            & \leq C(t - s)^\theta,
        \end{align*}
        for some $C = O(\Vert v \Vert_{L^\infty(0,T; H^2(\Omega)^2)})$.
    \end{proof}

    \subsection{Maximal regularity}
    We begin by
    \begin{proposition} \label{prop_semigroup_minus_I_int_At_Tts_fs_minus_ftau}
        Let $\sigma < \theta$ and $\mu_\ast/2 > \eta$.
        Let $0 < \varepsilon < 1$.
        Then
        \begin{align*}
            & \left \Vert
                (
                    e^{-(t - s) \mathcal{A}_{\mu, \delta}(s)}
                    - I
                )
                \int_0^s
                    \mathcal{A}_{\mu, \delta}(s) T_{\mu, \delta}(s, \tau)
                    (
                        f(s)
                        - f(\tau)
                    )
                d\tau
            \right \Vert\\
            & \leq C (t - s)^\sigma e^{- (\eta - \varepsilon) s} s^{- 1 + \beta - \sigma}.
        \end{align*}
        for all $f \in \mathcal{F}^{\eta, \sigma, \beta}(0, T; L^2_{\overline{\sigma}}(\Omega))$, where $C = O(\Vert v \Vert_{L^\infty(0,T; H^2(\Omega)^2)}^2)$.
    \end{proposition}
    \begin{proof}
        Using Proposition \ref{prop_estimate_I_minus_e_minus_t_s_with_integral} and \ref{eq_estimate_At_phi_Tts_minus_e_minus_t_s_At}, we estimate the left-hand side by
        \begin{align*}
            & \biggl \Vert
                (
                    e^{-(t - s) \mathcal{A}_{\mu, \delta}(s)}
                    - I
                ) \mathcal{A}_{\mu, \delta}(s)^{- \sigma}\\
            & \quad\quad\quad
                \times \int_0^s
                    \mathcal{A}_{\mu, \delta}(s)^{1 + \sigma} (
                        T_{\mu, \delta}(s, \tau)
                        - e^{-(s - \tau) \mathcal{A}_{\mu, \delta}(s)}
                    )
                    (
                        f(s)
                        - f(\tau)
                    )
                d\tau
            \biggr \Vert\\
            & + \left \Vert
                (
                    e^{-(t - s) \mathcal{A}_{\mu, \delta}(s)}
                    - I
                )
                \int_0^s
                    \mathcal{A}_{\mu, \delta}(s) e^{-(s - \tau) \mathcal{A}_{\mu, \delta}(s)}
                    (
                        f(s)
                        - f(\tau)
                    )
                d\tau
            \right \Vert\\
            & \leq C (t - s)^\sigma e^{- \eta s} (
                s^{- 1 - \sigma + \theta + \beta}
                + s^{- 1 + \beta - \sigma}
            )
            \\
            & \leq C (t - s)^\sigma e^{- (\eta - \varepsilon) s} s^{- 1 + \beta - \sigma},
        \end{align*}
        for $C = O(\Vert v \Vert_{L^\infty(0,T; H^2(\Omega)^2)}^2)$.
    \end{proof}

    \begin{lemma} \label{lem_maximal_regularity_estimate_in_F_nu_beta_sigma}
        Let $0 < \eta < \nu \leq \mu_\ast/2$, $\sigma < \theta < 1$, and $0 < \sigma < \beta < 1$.
        Let $V_0 \in D(\mathcal{A}(0)^\beta)$ and $F \in \mathcal{F}^{\nu, \sigma, \beta}(0, T; L^2_{\overline{\sigma}}(\Omega))$.
        Set
        \begin{align*}
            V(t)
            = T_{\mu, \delta}(t, 0) V_0
            + \int_0^t
                T_{\mu, \delta}(t,\tau) F(\tau)
            d\tau.
        \end{align*}
        Then
        \begin{align} \label{eq_maximal_regularity_estimate_in_F_nu_beta_sigma}
            \begin{split}
                & \sup_{0 < t < T} \Vert
                    \mathcal{A}_{\mu, \delta}(t)^\beta V(t)
                \Vert_{L^2(\Omega)^2}\\
                & + \left \Vert
                    \frac{du}{dt}
                \right \Vert_{\mathcal{F}^{\eta, \beta, \sigma}(0,T; L^2(\Omega)^2)}
                + \Vert
                    \mathcal{A}_{\mu, \delta}(t) u
                \Vert_{\mathcal{F}^{\eta, \beta, \sigma}(0,T; L^2(\Omega)^2)}\\
                & \leq C \left(
                    \Vert
                        \mathcal{A}(0)^\beta V_0
                    \Vert_{L^2(\Omega)^2}
                    + \Vert
                        F
                    \Vert_{\mathcal{F}^{\nu, \beta, \sigma}(0,T; L^2(\Omega)^2)}
                \right),
            \end{split}
        \end{align}
        where $C = O(\Vert v \Vert_{L^\infty(0,T; H^2(\Omega)^2)}^2)$.
    \end{lemma}
    \begin{proof}
        We first observe that
        \begin{align} \label{eq_split_At_ut}
            \begin{split}
                \mathcal{A}_{\mu, \delta}(t)^{\beta} V(t)
                & = \mathcal{A}_{\mu, \delta}(t)^\beta \mathcal{A}_{\mu, \delta}(0)^{- \beta} \mathcal{A}_{\mu, \delta}(0)^\beta V_0\\
                & + \int_0^t
                    \mathcal{A}_{\mu, \delta}(t)^\beta T_{\mu, \delta}(t, \tau) (
                        - F(t)
                        + F(\tau)
                    )
                d\tau\\
                & + \int_0^t
                    \mathcal{A}_{\mu, \delta}(t)^\beta (
                        T_{\mu, \delta}(t, \tau)
                        - e^{-(t - \tau)\mathcal{A}_{\mu, \delta}(t)}
                    )
                d\tau F(t)\\
                & + (
                    I
                    - e^{- t \mathcal{A}_{\mu, \delta}(t)}
                ) \mathcal{A}_{\mu, \delta}(t)^{-1 + \beta} F(t).
            \end{split}
        \end{align}
        Therefore, we deduce from Proposition \ref{eq_estimate_At_phi_Tts_minus_e_minus_t_s_At} that the left-hand side is bounded by
        \begin{align*}
            \text{(LHS)}
            & \leq C \Vert
                \mathcal{A}_{\mu, \delta}(0)^{\beta} V_0
            \Vert_{L^2(\Omega)^2}\\
            & + C \int_0^t
                (t - \tau)^{- \beta + \sigma} \tau^{-1 + \beta - \sigma} e^{- \nu \tau}
            d\tau
            \Vert
                f
            \Vert_{\mathcal{F}^{\nu, \sigma, \beta}(0,T; L^2(\Omega)^2)}\\
            & + C \int_0^t
                (t - \tau)^{- \beta + \theta} e^{- \mu_\ast (t - \tau)}
            d\tau
            \Vert
                f(t)
            \Vert_{L^2(\Omega)^2}\\
            & + C t^{1 - \beta}
            \Vert
                f(t)
            \Vert_{L^2(\Omega)^2},
        \end{align*}
        where $C = O(\Vert v \Vert_{L^\infty(0,T; H^2(\Omega)^2)})$.
        Since by change of variables, we have
        \begin{align*}
            \int_0^t
                (t - \tau)^{-\beta + \sigma} \tau^{-1 + \beta - \sigma} e^{-\nu \tau}
            d\tau
            \leq C
        \end{align*}
        and
        \begin{align*}
            \int_0^t
                (t - \tau)^{- \beta + \theta} e^{- \mu_\ast (t - \tau)}
            d\tau
            \int_0^t
                (t - \tau)^{- \beta} e^{- \mu_\ast (t - \tau)/2}
            d\tau
            \leq C t^{1 - \beta},
        \end{align*}
        we see that
        \begin{align*}
            & \sup_{0 < t < T} \Vert
                \mathcal{A}_{\mu, \delta}(t)^\beta V(t)
            \Vert_{L^2(\Omega)^2}\\
            & \leq C \left(
                \Vert
                    \mathcal{A}_{\mu, \delta}(0)^\beta V_0
                \Vert_{L^2(\Omega)^2}
                + \Vert
                    f
                \Vert_{\mathcal{F}^{\beta, \sigma}(0,T; L^2(\Omega)^2)}
            \right).
        \end{align*}
        We next observe the H\"{o}lder continuity.
        We split the difference as
        \begin{align*}
            & \Vert
                \mathcal{A}_{\mu, \delta}(t) T_{\mu, \delta}(t,0) V_0
                - \mathcal{A}_{\mu, \delta}(t) T_{\mu, \delta}(s, 0) V_0
            \Vert_{L^2(\Omega)^2}\\
            & \leq \Vert
                (
                    \mathcal{A}_{\mu, \delta}(t)T_{\mu, \delta}(t,s) \mathcal{A}_{\mu, \delta}(s)^{-1}\\
            & \quad\quad\quad\quad\quad
                    - e^{- (t - s) \mathcal{A}_{\mu, \delta}(s)}
                ) \mathcal{A}_{\mu, \delta}(s) T_{\mu, \delta}(s, 0) V_0
            \Vert_{L^2(\Omega)^2}\\
            & + \Vert
                (
                    e^{- (t - s) \mathcal{A}_{\mu, \delta}(s)}
                    - I
                ) \mathcal{A}_{\mu, \delta}(s) T_{\mu, \delta}(s, 0) V_0
            \Vert_{L^2(\Omega)^2}\\
            & =: I_1(t,s) + I_2(t,s).
        \end{align*}
        Propositions \ref{prop_estimate_for_At_phi_Tts_As_phi_phi_inverse} and \ref{prop_At_phi_Tts_As_minus_phi_minus_e_t_s_As} imply that
        \begin{align*}
            &I_1(t,s)\\
            & =
                \Vert
                    (
                        \mathcal{A}_{\mu, \delta}(t)T_{\mu, \delta}(t,s) \mathcal{A}_{\mu, \delta}(s)^{-1}
                        - e^{- (t - s) \mathcal{A}_{\mu, \delta}(s)}
                    )\\
            & \quad\quad\quad \times
                    \mathcal{A}_{\mu, \delta}(s) T_{\mu, \delta}(s, 0) \mathcal{A}_{\mu, \delta}(0)^{-\beta} \mathcal{A}_{\mu, \delta}(0)^\beta V_0
                \Vert_{L^2(\Omega)^2}\\
            & \leq C \vert
                t - s
            \vert^\theta
            e^{- \mu_\ast (t - s) /2 } s^{-(1 - \beta)} e^{- \mu_\ast s/2} \Vert
                \mathcal{A}_{\mu, \delta}(0)^\beta V_0
            \Vert_{L^2(\Omega)^2}\\
            & = C
            \vert
                t - s
            \vert^\theta
            e^{- \mu_\ast t /2 } s^{-(1 - \beta)} \Vert
                \mathcal{A}_{\mu, \delta}(0)^\beta V_0
            \Vert_{L^2(\Omega)^2}.
        \end{align*}
        and
        \begin{align*}
            &I_2(t,s)\\
            & = \Vert
                (
                    e^{- (t - s) \mathcal{A}_{\mu, \delta}(s)}
                    - I
                )\\
            & \quad\quad \times
                \mathcal{A}_{\mu, \delta}(s)^{-\sigma} \mathcal{A}_{\mu, \delta}(s)^{1 + \sigma} T_{\mu, \delta}(s, 0) \mathcal{A}_{\mu, \delta}(s)^{-\beta} \mathcal{A}_{\mu, \delta}(s)^\beta V_0
            \Vert_{L^2(\Omega)^2}\\
            & \leq C \vert
                t - s
            \vert^\sigma
            s^{-(1 + \sigma - \beta)} e^{- \mu_\ast s/2} \Vert
                \mathcal{A}_{\mu, \delta}(0)^\beta V_0
            \Vert_{L^2(\Omega)^2},
        \end{align*}
        where $C = O(\Vert v \Vert_{L^\infty(0,T; H^2(\Omega)^2)})$.
        Therefore, if $\sigma < \theta$ and $\eta \leq \mu_\ast/2$, we see that
        \begin{align*}
            e^{- \mu_\ast t/2 + \eta s /2}
            \leq e^{- (\mu_\ast t /2 - \eta s) /2} e^{- \mu_\ast t/4}
            \leq e^{- \mu_\ast (t - s) /4} e^{- \mu_\ast t/4},
        \end{align*}
        and then
        \begin{align*}
            & \sup_{0<s<t<T} e^{\eta s /2} s^{1 + \sigma - \beta}
            \frac{
                \Vert
                    \mathcal{A}_{\mu, \delta}(t) T_{\mu, \delta}(t,0) V_0
                    - \mathcal{A}_{\mu, \delta}(t) T_{\mu, \delta}(s, 0) V_0
                \Vert_{L^2(\Omega)^2}
            }{
                \vert
                    t - s
                \vert^\sigma
            }\\
            & \leq C \Vert
                \mathcal{A}_{\mu, \delta}(0)^\beta V_0
            \Vert_{L^2(\Omega)^2}(
                \sup_{\tau>0} \tau^{\theta - \sigma} e^{- \mu_\ast \tau/4}
                \sup_{\tau>0} e^{- \mu_\ast \tau/4} \tau^\sigma
                + 1
            )\\
            & \leq C \Vert
                \mathcal{A}_{\mu, \delta}(0)^\beta V_0
            \Vert_{L^2(\Omega)^2}.
        \end{align*}
        We next consider the forcing term.
        Similar to (\ref{eq_split_At_ut}), we observe that
        \begin{align*}
            & \int_0^t
                \mathcal{A}_{\mu, \delta}(t) T_{\mu, \delta}(t,\tau) F(\tau)
            d\tau\\
            & = \int_0^t
                \mathcal{A}_{\mu, \delta}(t) T_{\mu, \delta}(t, \tau) (
                    - F(t)
                    + F(\tau)
                )
            d\tau\\
            & + \int_0^t
                \mathcal{A}_{\mu, \delta}(t) (
                    T_{\mu, \delta}(t, \tau)
                    - e^{-(t- \tau) \mathcal{A}_{\mu, \delta}(t)}
                )
            d\tau F(t)
            + (
                I
                - e^{- t \mathcal{A}_{\mu, \delta}(t)}
            ) F(t)\\
            & =: J_1(t) + J_2(t) + J_3(t).
        \end{align*}
        By the assumption for $F$, we see that
        \begin{align*}
            \Vert
                J_1(t)
            \Vert_{L^2(\Omega)^2}
            & \leq C \int_0^t
                (t - \tau)^{-1 + \sigma} e^{- \mu_\ast (t - \tau)} \tau^{-1 + \beta - \sigma} e^{- \eta \tau}
            d\tau \Vert
                F
            \Vert_{\mathcal{F}^{\nu, \sigma, \beta}}\\
            & \leq C t^{- 1 + \beta} e^{- \eta t} \Vert
                F
            \Vert_{\mathcal{F}^{\nu, \sigma, \beta}},
        \end{align*}
        where $C = O(\Vert v \Vert_{L^\infty(0,T; H^2(\Omega)^2)})$.
        For the difference, we observe that
        \begin{align*}
            & J_1(t) - J_1(s)\\
            & = \int_s^t
                \mathcal{A}_{\mu, \delta}(t) T_{\mu, \delta}(t, \tau) (
                    - F(t)
                    + F(\tau)
                )
            d\tau\\
            & + \int_0^s
                \mathcal{A}_{\mu, \delta}(t) T_{\mu, \delta}(t, \tau) (
                    - F(t)
                    + F(\tau)
                )\\
            & \quad\quad\quad\quad\quad
                - \mathcal{A}_{\mu, \delta}(s) T_{\mu, \delta}(s, \tau) (
                    - F(s)
                    + F(\tau)
                )
            d\tau\\
            & =: J_{11}(t,s) + J_{12}(t,s).
        \end{align*}
        The H\"{o}lder continuity of $F$ implies that
        \begin{align*}
            \Vert
                J_{11}(t,s)
            \Vert_{L^2(\Omega)^2}
            & \leq \int_s^t
                (t - \tau)^{-1} e^{- \mu_\ast (t - \tau) - \nu \tau} \tau^{- 1 + \beta - \sigma} (
                    t - \tau
                )^{\sigma}
            d\tau
            \Vert
                f
            \Vert_{\mathcal{F}^{\nu, \sigma, \beta}}\\
            & \leq C (t - s)^\sigma s^{- 1 + \beta - \sigma} e^{- \nu t} \Vert
                f
            \Vert_{\mathcal{F}^{\nu, \sigma, \beta}}.
        \end{align*}
        where $C = O(\Vert v \Vert_{L^\infty(0,T; H^2(\Omega)^2)})$.
        We find that the following formula
        \begin{align*}
            & J_{12}(t,s)\\
            & = \int_0^s
                \mathcal{A}_{\mu, \delta}(t) T_{\mu, \delta}(t, \tau) (
                    - F(t)
                    + F(\tau)
                )\\
            & \quad\quad\quad\quad
                - \mathcal{A}_{\mu, \delta}(s) T_{\mu, \delta}(s, \tau) (
                    - F(s)
                    + F(\tau)
                )
            d\tau\\
            & = \int_0^s
                \mathcal{A}_{\mu, \delta}(t) e^{- (t - \tau) \mathcal{A}_{\mu, \delta}(t)} (
                    - F(t)
                    + F(s)
                )\\
            & \quad\quad
                + \mathcal{A}_{\mu, \delta}(t) (
                    T_{\mu, \delta}(t, \tau)
                    - e^{- (t - \tau) \mathcal{A}_{\mu, \delta}(t)}
                )
                (
                    - F(t)
                    + F(s)
                )\\
            & \quad\quad
                + \bigl (
                    \mathcal{A}_{\mu, \delta}(t) T_{\mu, \delta}(t, s) \mathcal{A}_{\mu, \delta}(s)^{-1}
                    - e^{-(t - s) \mathcal{A}_{\mu, \delta}(s)}
                \bigr )\\
            & \quad\quad\quad\quad\quad\quad \times
                \mathcal{A}_{\mu, \delta}(s) T_{\mu, \delta}(s, \tau)
                (
                    - F(s)
                    + F(\tau)
                )\\
            & \quad\quad
                + (
                    e^{-(t - s) \mathcal{A}_{\mu, \delta}(s)}
                    - I
                )
                \mathcal{A}_{\mu, \delta}(s) T_{\mu, \delta}(s, \tau)
                (
                    - F(s)
                    + F(\tau)
                )
            d\tau.
        \end{align*}
        Therefore we find from Propositions \ref{eq_estimate_At_phi_Tts_minus_e_minus_t_s_At}, \ref{prop_At_phi_Tts_As_minus_phi_minus_e_t_s_As} and \ref{prop_semigroup_minus_I_int_At_Tts_fs_minus_ftau} that
        \begin{align*}
            \Vert
                J_{12}(t,s)
            \Vert_{L^2(\Omega)^2}
            \leq C (t - s)^\sigma s^{-1 + \beta - \sigma} e^{- (\nu - \varepsilon) t} \Vert
                f
            \Vert_{\mathcal{F}^{\nu, \sigma, \beta}}.
        \end{align*}
        for some small $0 < \varepsilon < 1$ and constant $C = O(\Vert v \Vert_{L^\infty(0,T; H^2(\Omega)^2)}^2)$.
        We find Proposition \ref{eq_estimate_At_phi_Tts_minus_e_minus_t_s_At} that
        \begin{align*}
            \Vert
                J_2(t)
            \Vert_{L^2(\Omega)^2}
            & \leq C \int_0^t
                (t - \tau)^{-1 + \theta} e^{- \mu_\ast (t - \tau)}
            d\tau t^{-1 + \beta} e^{- \nu t} \Vert
                f
            \Vert_{\mathcal{F}^{\nu, \sigma, \beta}}\\
            & \leq C t^{-1 + \beta} e^{- (\nu - \varepsilon) t} \Vert
                f
            \Vert_{\mathcal{F}^{\nu, \sigma, \beta}}.
        \end{align*}
        For the difference, we observe that
        \begin{align*}
            & J_2(t) - J_2(s)\\
            & = \int_s^t
                \mathcal{A}_{\mu, \delta}(t) \tilde{W}_{\mu, \delta}(t, \tau)
            d\tau F(t)\\
            & + \int_0^s
                \mathcal{A}_{\mu, \delta}(t) \tilde{W}_{\mu, \delta}(t, \tau) F(t)
                - \mathcal{A}_{\mu, \delta}(s) \tilde{W}_{\mu, \delta}(s, \tau) F(s)
            d\tau\\
            & =: J_{21}(t,s) + J_{22}(t,s).
        \end{align*}
        We find from Proposition \ref{eq_estimate_At_phi_Tts_minus_e_minus_t_s_At} that
        \begin{align*}
            \Vert
                J_{21}(t, s)
            \Vert_{L^2(\Omega)^2}
            & \leq C (t - s)^\theta t^{-1 + \beta} e^{- \nu t} \Vert
                f
            \Vert_{\mathcal{F}^{\nu, \sigma, \beta}}\\
            & \leq C (t - s)^\sigma t^{-1 + \beta - \sigma} e^{- (\nu - \varepsilon) t} (t - s)^{\theta - \sigma} t^{\sigma} e^{- \varepsilon t} \Vert
                f
            \Vert_{\mathcal{F}^{\nu, \sigma, \beta}}\\
            & \leq C (t - s)^\sigma s^{-1 + \beta - \sigma} e^{- (\nu - \varepsilon) s} \Vert
                f
            \Vert_{\mathcal{F}^{\nu, \sigma, \beta}},
        \end{align*}
        where $C = O(\Vert v \Vert_{L^\infty(0,T; H^2(\Omega)^2)}^2)$.
        By the formula
        \begin{align*}
            & \mathcal{A}_{\mu, \delta}(t) \tilde{W}_{\mu, \delta}(t, \tau)
            - \mathcal{A}_{\mu, \delta}(s) \tilde{W}_{\mu, \delta}(s, \tau)\\
            & = \mathcal{A}_{\mu, \delta}(t) (
                 \tilde{W}_{\mu, \delta}(t, \tau)
                - T_{\mu, \delta}(t, s) \tilde{W}_{\mu, \delta}(s, \tau)
            )\\
            & + (
                \mathcal{A}_{\mu, \delta}(t) T_{\mu, \delta}(t, s) \mathcal{A}_{\mu, \delta}(s)^{-1}
                - e^{- (t - s) \mathcal{A}_{\mu, \delta}(s)}
            ) \mathcal{A}_{\mu, \delta}(s) \tilde{W}_{\mu, \delta}(s, \tau)\\
            & + (
                e^{- (t - s) \mathcal{A}_{\mu, \delta}(s)}
                - I
            ) \mathcal{A}_{\mu, \delta}(s)^{-\sigma}  \mathcal{A}_{\mu, \delta}(s)^{1 + \sigma} \tilde{W}_{\mu, \delta}(s, \tau)
        \end{align*}
        and Propositions \ref{eq_estimate_At_phi_Tts_minus_e_minus_t_s_At}, \ref{prop_At_phi_Tts_As_minus_phi_minus_e_t_s_As}, and \ref{prop_estimate_int_t_tildeWttau_minus_Tts_tildeWstau}, we deduce that
        \begin{align*}
            \Vert
                J_{22}(t,s)
            \Vert_{L^2(\Omega)^2}
            & \leq C (
                (t - s)^\sigma
                + (t - s)^\theta
            ) s^{-1 + \beta} e^{- \nu s} \Vert
                F
            \Vert_{\mathcal{F}^{\nu, \sigma, \beta}}\\
            & \leq C (t - s)^\sigma s^{-1 + \beta - \sigma} e^{- (\nu - \varepsilon) s} \Vert
                F
            \Vert_{\mathcal{F}^{\nu, \sigma, \beta}},
        \end{align*}
        where $C = O(\Vert v \Vert_{L^\infty(0,T; H^2(\Omega)^2)})$.
        It is clear that
        \begin{align*}
            \Vert
                J_3(t)
            \Vert_{L^2(\Omega)^2}
            \leq C t^{- 1 + \beta} e^{- \nu t} \Vert
                F
            \Vert_{\mathcal{F}^{\nu, \sigma, \beta}}.
        \end{align*}
        Since
        \begin{align*}
            J_3(t) - J_3(s)
            & = (
                F(t)
                - F(s)
            )
            + (
                e^{- t \mathcal{A}_{\mu, \delta}(t)} F(t)
                - e^{- s \mathcal{A}_{\mu, \delta}(s)} F(s)
            )\\
            & = (
                F(t)
                - F(s)
            )\\
            & + e^{- t \mathcal{A}_{\mu, \delta}(t)} (
                F(t)
                - F(s)
            )\\
            & + (
                e^{- t \mathcal{A}_{\mu, \delta}(t)}
                - e^{- t \mathcal{A}_{\mu, \delta}(s)}
            ) F(s)\\
            & + (
                e^{- (t - s) \mathcal{A}_{\mu, \delta}(s)}
                - I
            ) \mathcal{A}_{\mu, \delta}(s)^{-\sigma} \mathcal{A}_{\mu, \delta}(s)^\sigma e^{- s \mathcal{A}_{\mu, \delta}(s)} F(s),
        \end{align*}
        We deduce that
        \begin{align*}
            \Vert
                J_3(t) - J_3(s)
            \Vert_{L^2(\Omega)}
            & \leq C (t - s)^\sigma s^{-1 + \beta - \sigma} e^{- \nu s} \Vert
                F
            \Vert_{\mathcal{F}^{\nu, \sigma, \beta}},
        \end{align*}
        where $C = O(\Vert v \Vert_{L^\infty(0,T; H^2(\Omega)^2)})$.
        Therefore we conclude that
        \begin{align*}
            \Vert
                \mathcal{A}_{\mu, \delta}(t) V(t)
            \Vert_{\mathcal{F}^{\eta, \sigma, \beta}}
            \leq C \left(
                    \Vert
                        \mathcal{A}_{\mu, \delta}(0)^\beta u_0
                    \Vert_{L^2(\Omega)^2}
                    + \Vert
                        f
                    \Vert_{\mathcal{F}^{\nu, \beta, \sigma}(0,T; L^2(\Omega)^2)}
            \right),
        \end{align*}
        for $C = O(\Vert v \Vert_{L^\infty(0,T; H^2(\Omega)^2)}^2)$.
        Since $V(t)$ solves $dV(t)/dt + \mathcal{A}_{\mu, \delta}(t)V(t) = F(t)$ for $t>0$ in $L^2_{\overline{\sigma}}(\Omega)$, we obtain the estimate (\ref{eq_maximal_regularity_estimate_in_F_nu_beta_sigma}) for $dV(t)/dt$.
    \end{proof}
    In the complete same way we deduce
    \begin{lemma} \label{lem_maximal_regularity_estimate_in_F_beta_sigma}
        Let $\sigma < \theta$, and $\beta > \sigma$.
        Set
        \begin{align*}
            V_0
            \in D(\mathcal{A}_{\mu, \delta}(0)^\beta), \quad
            F
            \in \mathcal{F}^{\sigma, \beta}(0, T; L^2_{\overline{\sigma}}(\Omega)).
        \end{align*}
        Let
        \begin{align*}
            V(t)
            = T_{\mu, \delta}(t, 0) V_0
            + \int_0^t
                T_{\mu, \delta}(t,\tau) F(\tau)
            d\tau.
        \end{align*}
        Then
        \begin{align}
            \begin{split}
                & \sup_{0 < t < T} \Vert
                    \mathcal{A}_{\mu, \delta}(t)^\beta V(t)
                \Vert_{L^2(\Omega)^2}\\
                & + \left \Vert
                    \frac{dV}{dt}
                \right \Vert_{\mathcal{F}^{\beta, \sigma}(0,T; L^2(\Omega)^2)}
                + \Vert
                    \mathcal{A}_{\mu, \delta}(t) V
                \Vert_{\mathcal{F}^{\beta, \sigma}(0,T; L^2(\Omega)^2)}\\
                & \leq C_T \left(
                    \Vert
                        \mathcal{A}_{\mu, \delta}(0)^\beta V_0
                    \Vert_{L^2(\Omega)^2}
                    + \Vert
                        F
                    \Vert_{\mathcal{F}^{\beta, \sigma}(0,T; L^2(\Omega)^2)}
                \right).
            \end{split}
        \end{align}
        for $C = O(\Vert v \Vert_{L^\infty(0,T; H^2(\Omega)^2)}^2)$.
    \end{lemma}
    Note that the constant in Lemma \ref{lem_maximal_regularity_estimate_in_F_beta_sigma} $C$ depends on $T$.
    The growth order of $T$ is at most first order due to the calculations in the proof of Lemma \ref{lem_maximal_regularity_estimate_in_F_nu_beta_sigma}.

\section{Nonlinear problems} \label{section_nonlinear_problems}
\subsection{Construction of the solution}
Let $t_n$ be the time defined in (\ref{eq_definition_of_t0}) for some large $n$.
At this time we can assume that $\nabla V(t_n)$ is sufficiently small.
For vector fields $V_1$ and $V_2$, we set
\begin{align}
    \begin{split}
        N(V_1, V_2)
        & = 
            (
                V_1(\tau) \cdot \nabla_H V_2(\tau)
                + W_1 (\tau)\partial_3 V_2(\tau)
            )
        \\
    \end{split}
\end{align}
From simple calculations, we can see that
\begin{align} \label{eq_diff_Nv1v1_Nv2v2}
    N(V_1, V_1)
    - N(V_2, V_2)
    = N(\tilde{V}, \overline{V})
    + N(\overline{V}, \tilde{V})
\end{align}
where $\tilde{V} = V_1 - V_2$ and $\overline{V} = (V_1 + V_2)/2$.
We next consider the correspond integral equations to (\ref{eq_diff}) such that
\begin{align} \label{eq_integral_equation_of_V}
    V(t)
    = V(s)
    + \int_s^t
        T_{\mu, \delta}(t, \tau) P (
            N(V)
            + F_\delta(\tau)
        )
    d\tau
    =: M_{\mu, \delta}(s, t; V)
\end{align}
for $0 < s < t$.
We denote the complex interpolation space $X_\alpha \in L^2_{\overline{\sigma}}$ for $\alpha \in [0,1]$ by
\begin{align*}
    X_\alpha
    = [L^2_{\overline{\sigma}}(\Omega), D(A)]_{\alpha}.
\end{align*}
Note that $X_\alpha \hookrightarrow H^{2\alpha}(\Omega)$ and
\begin{align*}
    X_{1/2}
    = \{ v \in H^1_{\overline{\sigma}}(\Omega)
        ;
        \mathrm{div}_H \overline{v} = 0 \text{ in $\Omega$}, \,
        v_{z=-l} \equiv 0
    \},
\end{align*}
see also Section 4 in \cite{HieberKashiwabara2016}.
Hereafter, we consider the solution to equation (\ref{eq_diff}) starting from the initial data $V(t_n)$. 
For notational simplicity, we shift the time variable so that $t_n$ becomes the initial time.
To simplify notation further, we will assume $t_n=0$.
Consequently, due to the definition of $t_n$, the corresponding initial data $V_0$ is assumed to be sufficiently small in $H^1_{\overline{\sigma}}(\Omega)$.

\subsection{The case \texorpdfstring{$F_\delta = 0$}{F\_delta = 0}}
We first consider the case $F_\delta = 0$.
As mentioned before, this simpler case corresponds to a scenario where the exact external force is known without the need for observation.
\begin{lemma}\label{lem_small_data_wellposedness_F_eta_sigma_beta_settings}
    Let $0 < \theta < \alpha$.
    Let $v \in C^{1,\theta}(0, \infty; L^2_{\overline{\sigma}}(\Omega)) \cap C^\theta(0, \infty; H^2(\Omega))$ be the solution to (\ref{eq_primitive}).
    Assume that $\mu>0$, $\delta>0$, and $v$ satisfy (\ref{eq_condition_mu_and_delta_and_norm_of_v_1}) and (\ref{eq_condition_mu_and_delta_and_norm_of_v_2}).
    Then, there exists $\varepsilon_0$ such that if
    \begin{align*}            
        \Vert
            V_0
        \Vert_{X_{1/2}}
        \leq \varepsilon_0,
    \end{align*}
    there exists a unique solution $V \in \mathcal{F}^{\eta, \sigma, \beta}(0,\infty; L^2_{\overline{\sigma}}(\Omega))$ to (\ref{eq_integral_equation_of_V}) such that
    \begin{align*}
        \Vert
            V
        \Vert_{\mathcal{F}^{\mu_\ast/2, \sigma, \beta}(0,\infty; L^2_{\overline{\sigma}}(\Omega))}
        \leq C \Vert
            V_0
        \Vert_{X_{1/2}}
    \end{align*}
    for some constant $C>0$.
\end{lemma}
\begin{proof}
    By the assumptions on $v, \mu, \delta$, we have evolution operator $T_{\mu, \delta} (t,s)$ satisfying (\ref{eq_estimate_for_T_t_s_by_interpolation}).
    The proof is similar to that of Proposition 5.2 in \cite{HieberKashiwabara2016}.
    We prove the existence of the solution using the Fujita-Kato method.
    Let $R>0$ be sufficiently small, to be determined later.
    Assume that $V$ lies in a small ball $B_{\mathcal{F}^{\mu_\ast/2, \sigma, \beta}(0,\infty; L^2_{\overline{\sigma}}(\Omega))}(R)$.
    We find from Proposition \ref{prop_estimate_for_convection_term} and the interpolation inequality
    \begin{align*}
        \Vert
            \varphi
        \Vert_{X_{3/4}}
        \leq \Vert
            \varphi
        \Vert_{X_{1/2}}^{1/2}
        \Vert
            \varphi
        \Vert_{D(A)}^{1/2}
    \end{align*}
    that
    \begin{equation} \label{eq_estimate_for_N_0_t_in_X_12}
        \begin{split}
            & e^{\mu_\ast s} s^{1/2} \Vert
                N(V(s), V(s))
            \Vert_{L^2(\Omega)^2} \\
            & \leq C e^{\mu_\ast s} s^{1/2} s^{1/2} \left(
                \Vert
                    V (s) \cdot \nabla_H V (s)
                \Vert_{L^2(\Omega)}
                + \Vert
                    W (s) \partial_3 V (s)
                \Vert_{L^2(\Omega)}
            \right)\\
            & \leq C e^{\mu_\ast s} s^{1/2} \Vert
                V (s)
            \Vert_{X_{3/4}}^2\\
            & \leq C e^{\mu_\ast s} s^{1/2} \Vert
                V (s)
            \Vert_{X_{1/2}}
            \Vert
                V (s)
            \Vert_{D(\mathcal{A}_{\mu, \delta}(s))}\\
            & \leq C \Vert
                V (s)
            \Vert_{\mathcal{F}^{\mu_\ast/2, \sigma, 1/2}(0,\infty; L^2_{\overline{\sigma}}(\Omega))}^2.
        \end{split}
    \end{equation}
    Similar to this, using (\ref{eq_diff_Nv1v1_Nv2v2}), we find that
    \begin{align*}
        & e^{\mu_\ast s} s^{1/2 - \sigma}
        \Vert
            N(V(t), V(t))
            - N(V(s), V(s))
        \Vert_{L^2(\Omega)^2} \\
        & \leq C \vert
            t
            - s
        \vert^\sigma
        \Vert
            V
        \Vert_{\mathcal{F}^{\mu_\ast/2, \sigma, 1/2}(0,\infty; L^2_{\overline{\sigma}}(\Omega))}^2.
    \end{align*}
    Applying Lemma \ref{lem_maximal_regularity_estimate_in_F_nu_beta_sigma}, we obtain the quadratic inequality
    \begin{align} \label{eq_estimate_for_the_RHS_of_integral_equation_of_eq_diff}
        \begin{split}
            & \left \Vert
                \frac{d}{dt} M_{\mu, \delta}(0, \cdot; V)
            \right \Vert_{\mathcal{F}^{\mu_\ast/2, \sigma, 1/2}(0,\infty; L^2_{\overline{\sigma}}(\Omega))} \\
            & + \Vert
                \mathcal{A}_{\mu, \delta} (t)M_{\mu, \delta}(0, \cdot; V)
            \Vert_{\mathcal{F}^{\mu_\ast/2, \sigma, 1/2}(0,\infty; L^2_{\overline{\sigma}}(\Omega))} \\
            & \leq C_0 \Vert
                V_0
            \Vert_{X_{1/2}}
            + C_1 \Vert
                V
            \Vert_{\mathcal{F}^{\mu_\ast/2, \sigma, 1/2}(0,\infty; L^2_{\overline{\sigma}}(\Omega))}^2.
        \end{split}
    \end{align}
    The quadratic estimate implies that, if we take $\Vert V_0 \Vert_{X_{1/2}}$ so small that
    \begin{align*}
        r
        & := 4 C_0 C_1 \Vert
            V_0
        \Vert_{X_{1/2}}
        \leq \frac{1}{2},
    \end{align*}
    then we have
    \begin{align} \label{eq_bound_for_size_of_V_for_fujita_kato_method}
        \begin{split}
            & \sup_{0 < t < \infty} \Vert
                \mathcal{A}_{\mu, \delta}(t)^{1/2} M_{\mu, \delta}(0, \cdot; V)
            \Vert_{L^2(\Omega)^2}\\
            & + \left \Vert
                \frac{d}{dt} M_{\mu, \delta}(0, \cdot; V)
            \right \Vert_{\mathcal{F}^{\mu_\ast/2, \sigma, 1/2}(0,\infty; L^2_{\overline{\sigma}}(\Omega))} \\
            & + \Vert
                \mathcal{A}_{\mu, \delta} (t)M_{\mu, \delta}(0, \cdot; V)
            \Vert_{\mathcal{F}^{\mu_\ast/2, \sigma, 1/2}(0,\infty; L^2_{\overline{\sigma}}(\Omega))} \\
            & \leq \frac{
                1 - \sqrt{
                    1 - r
                }
            }{2 C_1}\\
            & =: R
            \leq 2 C_0 \Vert
                V_0
            \Vert_{X_{1/2}}.
        \end{split}
    \end{align}
    Let $V_1, V_2 \in B_{\mathcal{F}^{\mu_\ast/2, \sigma, \beta}(0,\infty; L^2_{\overline{\sigma}}(\Omega))}(R)$ for some small $R>0$.
    We see from the identity (\ref{eq_diff_Nv1v1_Nv2v2}) and Proposition \ref{prop_estimate_for_convection_term} that
    \begin{align} \label{eq_Mmudelta_is_self_mapping}
        \begin{split}
            & e^{\mu_\ast s} s^{1/2} \Vert
                N(V_1(s), V_1(s))
                - N(V_2(s), V_2(s))
            \Vert_{L^2_{\overline{\sigma}}(\Omega)}\\
            & \leq C \Vert
                V_1
                - V_2
            \Vert_{\mathcal{F}^{\mu_\ast/2, \sigma, 1/2}(0,\infty; L^2_{\overline{\sigma}}(\Omega))}
            \sum_{j=1,2}
                \Vert
                    V_j
                \Vert_{\mathcal{F}^{\mu_\ast/2, \sigma, 1/2}(0,\infty; L^2_{\overline{\sigma}}(\Omega))}.
        \end{split}
    \end{align}
    We observe from (\ref{eq_diff_Nv1v1_Nv2v2}) that
    \begin{align*}
        & \bigl (
                N(V_1(t), V_1(t))
                - N(V_2(t), V_2(t))
            \bigr )\\
        & \quad\quad\quad
            - \bigl (
                N(V_1(s), V_1(s))
                - N(V_2(s), V_1(s))
            \bigr )\\
        & = N(\tilde{V}(t), \overline{V}(t))
        + N(\overline{V}(t), \tilde{V}(t))
        - N(\tilde{V}(s), \overline{V}(s))
        - N(\overline{V}(s), \tilde{V}(s))\\
        & = N(\tilde{V}(t) - \tilde{V}(s), \overline{V}(t))
        + N(\tilde{V}(s), \overline{V}(t) - \overline{V}(s))\\
        & + N(\overline{V}(t) - \overline{V}(s), \tilde{V}(t))
        + N(\overline{V}(s), \tilde{V}(t) - \tilde{V}(s)).
    \end{align*}
    Therefore we find that
    \begin{align*} 
        \begin{split}
            & e^{\mu_\ast s} s^{1/2 - \sigma} \Vert
                N(V_1(t))
                - N(V_2(t))
                - (
                    N(V_1(s))
                    - N(V_2(s))
                )
            \Vert_{L^2_{\overline{\sigma}}(\Omega)}\\
            & \leq C \vert
                t - s
            \vert^\sigma \Vert
                V_1
                - V_2
            \Vert_{\mathcal{F}^{\mu_\ast/2, \sigma, 1/2}(0,\infty; L^2_{\overline{\sigma}}(\Omega))}\\
            & \quad\quad\quad
            \times \sum_{j=1,2}
                \Vert
                    V_j
                \Vert_{\mathcal{F}^{\mu_\ast/2, \sigma, 1/2}(0,\infty; L^2_{\overline{\sigma}}(\Omega))}.
        \end{split}
    \end{align*}
    We obtain
    \begin{align} \label{eq_Mmudelta_is_contraction_mapping}
        \begin{split}
            & \sup_{0 < t < \infty} \Vert
                \mathcal{A}_{\mu, \delta}(t)^{1/2} (
                    M_{\mu, \delta}(0, \cdot; V_1)
                    - M_{\mu, \delta}(0, \cdot; V_2)
                )
            \Vert_{L^2(\Omega)^2}\\
            & + \left \Vert
                \frac{d}{dt} (
                    M_{\mu, \delta}(0, \cdot; V_1)
                    - M_{\mu, \delta}(0, \cdot; V_2)
                )
            \right \Vert_{\mathcal{F}^{\mu_\ast/2, \sigma, 1/2}(0,\infty; L^2_{\overline{\sigma}}(\Omega))}\\
            & + \Vert
                \mathcal{A}_{\mu, \delta}(
                    M_{\mu, \delta}(0, \cdot; V_1)
                    - M_{\mu, \delta}(0, \cdot; V_2)
                )
            \Vert_{\mathcal{F}^{\mu_\ast/2, \sigma, 1/2}(0,\infty; L^2_{\overline{\sigma}}(\Omega))}\\
            & \leq C_3 \Vert
                V_1
                - V_2
            \Vert_{\mathcal{F}^{\mu_\ast/2, \sigma, 1/2}(0,\infty; L^2_{\overline{\sigma}}(\Omega))}\\
            & \quad\quad\quad
            \times \sum_{j=1,2}
            \Vert
                V_j
            \Vert_{\mathcal{F}^{\mu_\ast/2, \sigma, 1/2}(0,\infty; L^2_{\overline{\sigma}}(\Omega))}.
        \end{split}
    \end{align}
    If $R$ is so small that
    \begin{align*}
        C_3 \sum_{j=1,2}
        \Vert
            V_j
        \Vert_{\mathcal{F}^{\mu_\ast/2, \sigma, 1/2}(0,\infty; L^2_{\overline{\sigma}}(\Omega))}
        \leq 2R
        < 1,
    \end{align*}
    the solution map
    \begin{align*}
        V
        \longmapsto V_0
        + M_{\mu, \delta}(0, \cdot; V)
    \end{align*}
    is a contractive in $\mathcal{F}^{\mu_\ast/2, \sigma, 1/2}(0,\infty; L^2_{\overline{\sigma}}(\Omega))$ for each small $V_0 \in X_{1/2}$.
    Therefore, Banach's fixed point method implies that there exist a unique solution to 
    \begin{align*}
        V(t)
        = V_0
        + M_{\mu, \delta}(0,t; V)
    \end{align*}
    such that
    \begin{align*}
        \Vert
            V
        \Vert_{\mathcal{F}^{\mu_\ast/2, \sigma, 1/2}(0,\infty; L^2_{\overline{\sigma}}(\Omega))}
        \leq R
    \end{align*}
    for all $t>0$.
\end{proof}

\begin{proof}[Proof of Theorem \ref{thm_main_thoerem1}.]
    Lemma \ref{lem_small_data_wellposedness_F_eta_sigma_beta_settings} implies (\ref{eq_decay_estimate_in_main_theorem}).
\end{proof}

\subsection{The case \texorpdfstring{$F_\delta$}{F\_delta} is small}
We assume that $F_\delta = f - J_\delta f$ and then
\begin{align*}
    \Vert
        F_\delta (t)
    \Vert
    \leq C \delta \sup_{t>0} \Vert
        \nabla f(t)
    \Vert_{L^2(\Omega)^2}
\end{align*}
for some constant $C>0$ and all $t>0$.
We take $T = 2$.
Since the initial data and external force can be taken small, we can construct the solution up to $T=2$.
Actually, we can employ the similar way as in (\ref{eq_Mmudelta_is_self_mapping}) and Lemma \ref{lem_maximal_regularity_estimate_in_F_beta_sigma} to see that
\begin{align} \label{eq_estimate_for_M_mu_delta_in_F_sigma_half}
    \begin{split}
        & \sup_{0 < t < \infty} \Vert
            \mathcal{A}_{\mu, \delta}(t)^{1/2} M_{\mu, \delta}(0, \infty; V)
        \Vert_{L^2(\Omega)^2}\\
        & + \left \Vert
            \frac{d}{dt} M_{\mu, \delta}(0, \cdot; V)
        \right \Vert_{\mathcal{F}^{\sigma, 1/2}(0, T; L^2_{\overline{\sigma}}(\Omega))} \\
        & + \Vert
            \mathcal{A}_{\mu, \delta} (t)M_{\mu, \delta}(0, \cdot; V)
        \Vert_{\mathcal{F}^{\sigma, 1/2}(0, T; L^2_{\overline{\sigma}}(\Omega))}\\
        & \leq \frac{
            1 - \sqrt{
                1 - r
            }
        }{2 C_1}
        =: R,
    \end{split}
\end{align}
for all $T>0$ and
\begin{align*}
    r
    = 4 C_0 C_1 (
        \Vert
            V_0
        \Vert_{X_{1/2}}
        + \delta \Vert
            \nabla f
        \Vert_{\mathcal{F}^{\sigma, 1/2}(0, T; L^2_{\overline{\sigma}}(\Omega))}
    ).
\end{align*}
Therefore combining with (\ref{eq_Mmudelta_is_contraction_mapping}), we find the solution $V \in \mathcal{F}^{\sigma, 1/2}(0, T; L^2_{\overline{\sigma}}(\Omega))$ to (\ref{eq_diff}) associated with small initial data $V_0$.
Let $N_0$ be sufficiently large.
In view of the estimate (\ref{eq_definition_of_t0}), we find a sequence $N_0 < t_0 < t_1 < \cdots$ such that
\begin{align*}
    \vert
        t_{j+1}
        - t_j
    \vert < 2, \quad
    j \in \Integer_{\geq 0}
\end{align*}
and (\ref{eq_definition_of_t0}) is satisfied.
Therefore repeating (\ref{eq_estimate_for_M_mu_delta_in_F_sigma_half}) in the time interval $(t_j, t_j + 1)$, we obtain
\begin{align} 
    \begin{split}
        & \sup_{0 < t < \infty} \Vert
            \mathcal{A}_{\mu, \delta}(t)^{1/2} M_{\mu, \delta}(0, \infty; V)
        \Vert_{L^2(\Omega)^2}\\
        & + \left \Vert
            \frac{d}{dt} M_{\mu, \delta}(0, \infty; V)
        \right \Vert_{\mathcal{F}^{\sigma, 1/2}(0, \infty; L^2_{\overline{\sigma}}(\Omega))} \\
        & + \Vert
            \mathcal{A}_{\mu, \delta} (t)M_{\mu, \delta}(0, \cdot; V)
        \Vert_{\mathcal{F}^{\sigma, 1/2}(0, \infty; L^2_{\overline{\sigma}}(\Omega))}\\
        & \leq R\\
        & \leq 2 C_0 \left(
            \varepsilon
            + \frac{C \delta}{\mu^{1/2}} \sup_{s>0}
            \Vert
                \nabla f (s)
            \Vert_{L^2(\Omega)}
            + \delta \Vert
                \nabla f (s)
            \Vert_{L^2(\Omega)}
        \right).
    \end{split}
\end{align}
In the last inequality we used (\ref{eq_definition_of_t0}). 
We have proved Theorem \ref{thm_main_thoerem2}.

\section*{Acknowledgments}
The author would like to thank the members of RIKEN Pioneering Project “Prediction for Science” for helpful discussions and comments for this research.
The author was partly supported by RIKEN Pioneering Project ``Prediction for Science'' and JSPS Grant-in-Aid for Young Scientists (No. 22K13948).

\end{document}